\documentclass[a4paper, 11pt, DIV10, headinclude=false, footinclude=false]{scrartcl}
\title{Numerical upscaling for wave equations with time-dependent multiscale coefficients%
\thanks{This work is funded by the German Research Foundation (DFG) -- Project-ID 258734477 -- SFB 1173 as well as by the Federal Ministry of Education and Research (BMBF) and the Baden-W\"urttemberg Ministry of Science as part of the Excellence Strategy of the German Federal and State Governments.}
}

\usepackage{cite}
\usepackage{amsmath}
\usepackage{amssymb}
\usepackage{amsthm}
\usepackage{mathtools}
\usepackage{calc}
\usepackage{ifthen}
\usepackage{hyperref}
\usepackage[pdftex,dvipsnames]{xcolor}
\usepackage{pgfplots}
\usepackage[nameinlink]{cleveref}
\usepackage{autonum} 
\crefname{equation}{}{}
\crefname{enumi}{}{}
\crefdefaultlabelformat{#2\textup{#1}#3}
\Crefformat{corollary}{#2\normalfont{Corollary}~#1#3}

\newtheorem{theorem}{Theorem}
\newtheorem{proposition}[theorem]{Proposition}
\newtheorem{lemma}[theorem]{Lemma}
\newtheorem{corollary}[theorem]{Corollary}
\newtheorem{remark}[theorem]{Remark}

\numberwithin{theorem}{section}
\numberwithin{figure}{section}
\numberwithin{equation}{section}

\newenvironment{abstr}[1]{ \vspace{.05in}\footnotesize
	\parindent .2in
	{\upshape\bfseries #1. }\ignorespaces}{\par\vspace{.1in}}
\newenvironment{Abstract}{\begin{abstr}{Abstract}}{\end{abstr}}
\newenvironment{keywords}{\begin{abstr}{Key words}}{\end{abstr}}
\newenvironment{AMS}{\begin{abstr}{AMS subject classifications}}{\end{abstr}}

\DeclarePairedDelimiter{\abs}{\lvert}{\rvert}
\DeclarePairedDelimiter{\norm}{\lVert}{\rVert}

\newcommand{\sprod}[3]{\left( #2 \mid #3 \right)_{#1}}
\newcommand{\expp}[1]{\mathrm{e}^{#1}}

\newcommand{\diff}{\mathop{}\!\mathrm{d}}
\newcommand{\dt}{\diff t}
\newcommand{\dx}{\diff x}

\let\div\undefined
\DeclareMathOperator{\div}{\nabla \cdot} 
\DeclareMathOperator{\grad}{\nabla} 
\DeclareMathOperator{\Id}{\operatorname{Id}}
\newcommand{\Lin}[1]{\mathcal{L}(#1)}

\newcommand{\R}{\mathbb{R}} 
\newcommand{\N}{\mathbb{N}} 

\newcommand{\pt}{\partial_t}
\newcommand{\ptt}{\partial_{tt}}
\newcommand{\pttt}{\partial_{ttt}}
\newcommand{\domain}{\Omega}
\newcommand{\spacedim}{d}

\newcommand{\eps}{\varepsilon}

\newcommand{\Triah}{\mathcal{T}_h}
\newcommand{\TriaH}{\mathcal{T}_H}
\newcommand{\patch}[2][]{N_{\ifthenelse{\equal {#1} {}}{k}{#1}}{\ifthenelse{\equal {#2} {}}{}{(#2)}}}

\newcommand{\acoeff}{a}
\newcommand{\aeps}{\acoeff_\eps}
\newcommand{\aepsd}[1]{\acoeff_\eps^{#1}}
\newcommand{\apatch}[2][]{\overline{a}_{\patch[#1]{#2}}}
\newcommand{\alc}{c_a}
\newcommand{\auc}{C_a}
\newcommand{\ptauc}{C_a^{\pt}}

\newcommand{\tor}[1][]{\ifthenelse{\equal {#1} {}}{(t)}{({#1})}}
\newcommand{\f}{\textnormal{f}}
\newcommand{\ms}{\textnormal{ms}}
\newcommand{\patchsize}{k}
\newcommand{\kms}{{\patchsize,\ms}}

\newcommand{\Vhyp}{\mathcal{V}}
\newcommand{\Vhypt}[1][]{\Vhyp\tor[#1]}
\newcommand{\Hhyp}{\mathcal{H}}
\newcommand{\Xhyp}{\mathcal{X}}
\newcommand{\Xhypt}[1][]{\Xhyp\tor[#1]}

\newcommand{\solhypsec}{u}
\newcommand{\ptsolhypsec}{v}
\newcommand{\Ophypsec}{A}
\newcommand{\rhshypsec}{f}
\newcommand{\solhypseceps}{\solhypsec_\eps}

\newcommand{\solhypsecinit}{\solhypsec^0}
\newcommand{\ptsolhypsecinit}{\ptsolhypsec^0}

\newcommand{\solhyp}{y}
\newcommand{\Ophyp}{\mathcal{A}}
\newcommand{\rhshyp}{\mathcal{F}}
\newcommand{\solhypeps}{\solhyp_\eps}
\newcommand{\Ophypeps}{\Ophyp_\eps}
\newcommand{\solhypinit}{\solhyp^0}

\newcommand{\FESpaceh}{W_h}
\newcommand{\FESpaceH}{W_H}

\newcommand{\Vhyph}{\Vhyp_h}

\newcommand{\Hhyph}{\Hhyp_h}
\newcommand{\Xhyph}{\Xhyp_h}

\newcommand{\VhypH}{\Vhyp_H}
\newcommand{\HhypH}{\Hhyp_H}
\newcommand{\XhypH}{\Xhyp_H}

\newcommand{\Vhypf}{\Vhyp^f}
\newcommand{\Vhypms}{\Vhyp_\ms}
\newcommand{\Vhypmst}[1][]{\Vhypms\tor[#1]}
\newcommand{\Vhypmsk}{\Vhyp_\kms}
\newcommand{\Vhypmskt}[1][]{\Vhypmsk\tor[#1]}

\newcommand{\Hhypmsk}{\Hhyp_\kms}
\newcommand{\Hhypmskt}[1][]{\Hhypmsk\tor[#1]}
\newcommand{\Xhypms}{\Xhyp_\ms}

\newcommand{\Xhypmsk}{\Xhyp_\kms}
\newcommand{\Xhypmskt}[1][]{\Xhypmsk\tor[#1]}

\newcommand{\InterpolH}{I_H}
\newcommand{\ProjVhf}{R_\f}
\newcommand{\ProjVhfK}{R_{\f,K}}
\newcommand{\ProjVhft}[1][]{\ProjVhf\tor[#1]}
\newcommand{\ProjVhfd}[1][]{\ProjVhf^{\ifthenelse{\equal {#1} {}}{}{{#1}}}}
\newcommand{\ProjVhfKt}[1][]{R_{\f,K}\tor[#1]}
\newcommand{\ProjVhms}{R_\ms}
\newcommand{\ProjVhmst}[1][]{\ProjVhms\tor[#1]}
\newcommand{\ProjVhfk}{R_{\f,k}}
\newcommand{\ProjVhfkK}{R_{\f,k,K}}
\newcommand{\ProjVhfkt}[1][]{\ProjVhfk\tor[#1]}
\newcommand{\ProjVhfkd}[1][]{\ProjVhfk^{\ifthenelse{\equal {#1} {}}{}{{#1}}}}
\newcommand{\ProjVhfkKt}[1][]{\ProjVhfkK\tor[#1]}
\newcommand{\ProjVhfkKd}[1][]{\ProjVhfkK^{\ifthenelse{\equal {#1} {}}{}{{#1}}}}
\newcommand{\ProjVhmsk}{R_\kms}
\newcommand{\ProjVhmskt}[1][]{\ProjVhmsk\tor[#1]}
\newcommand{\ProjVhmskd}[1][]{\ProjVhmsk^{\ifthenelse{\equal {#1} {}}{}{{#1}}}}

\newcommand{\ProjHhmsk}{P_\kms}
\newcommand{\ProjHhmskt}[1][]{\ProjHhmsk\tor[#1]}
\newcommand{\ProjVhmsktilde}{\widetilde{R}_\kms}
\newcommand{\ProjVhmsktildet}[1][]{\ProjVhmsktilde\tor[#1]}
\newcommand{\ProjVhmsktilded}[1][]{\ProjVhmsktilde^{\ifthenelse{\equal {#1} {}}{}{{#1}}}}

\newcommand{\timed}[1]{t_{#1}} 

\newcommand{\ProjXhmsk}{\mathcal{P}_\kms}
\newcommand{\RefXhmsktilde}{\widetilde{\mathcal{R}}_\kms}

\newcommand{\RefXhmsktildet}[1][]{\RefXhmsktilde\tor[#1]}

\newcommand{\solhyph}{\solhyp_h}
\newcommand{\solhypsech}{\solhypsec_h}
\newcommand{\ptsolhypsech}{\ptsolhypsec_h}
\newcommand{\Ophyph}{\Ophyp_h}
\newcommand{\Ophypsech}{\Ophypsec_h}
\newcommand{\rhshyph}{\rhshyp_h}
\newcommand{\rhshypsech}{\rhshypsec_h}
\newcommand{\solhyphinit}{\solhyp_h^0}
\newcommand{\solhypsechinit}{\solhypsec_h^0}
\newcommand{\ptsolhypsechinit}{\ptsolhypsec_h^0}

\newcommand{\solhypmsk}{\solhyp_\kms}
\newcommand{\solhypsecmsk}{\solhypsec_\kms}
\newcommand{\ptsolhypsecmsk}{\ptsolhypsec_\kms}
\newcommand{\Ophypmsk}{\Ophyp_\kms}
\newcommand{\rhshypmsk}{\rhshyp_\kms}
\newcommand{\solhypmskinit}{\solhyp_\kms^0}

\newcommand{\errmsk}{e_\kms}
\newcommand{\errmsseck}{e_{\kms,\solhypsec}}
\newcommand{\pterrmsseck}{e_{\kms,\ptsolhypsec}}
\newcommand{\errmskrhs}{G_\kms}

\newcommand{\Xhypd}[1][]{\Xhyp^{\ifthenelse{\equal {#1} {}}{}{{#1}}}}
\newcommand{\Xhypmsd}[1][]{\Xhypms^{\ifthenelse{\equal {#1} {}}{}{{#1}}}}
\newcommand{\Xhypmskd}[1][]{\Xhypmsk^{\ifthenelse{\equal {#1} {}}{}{{#1}}}}
\newcommand{\Vhypd}[1][]{\Vhyp^{\ifthenelse{\equal {#1} {}}{}{{#1}}}}
\newcommand{\Vhypmsd}[1][]{\Vhypms^{\ifthenelse{\equal {#1} {}}{}{{#1}}}}
\newcommand{\Vhypmskd}[1][]{\Vhypmsk^{\ifthenelse{\equal {#1} {}}{}{{#1}}}}
\newcommand{\Hhypmskd}[1][]{\Hhypmsk^{\ifthenelse{\equal {#1} {}}{}{{#1}}}}

\newcommand{\solhypmskd}[1][]{\solhypmsk^{\ifthenelse{\equal {#1} {}}{}{{#1}}}}
\newcommand{\solhypsecmskd}[1][]{\solhypsecmsk^{\ifthenelse{\equal {#1} {}}{}{{#1}}}}
\newcommand{\ptsolhypsecmskd}[1][]{\ptsolhypsecmsk^{\ifthenelse{\equal {#1} {}}{}{{#1}}}}
\newcommand{\RefXhmsktilded}[1][]{\RefXhmsktilde^{\ifthenelse{\equal {#1} {}}{}{{#1}}}}
\newcommand{\ProjXhmskd}[1][]{\ProjXhmsk^{\ifthenelse{\equal {#1} {}}{}{{#1}}}}
\newcommand{\Ophypmskd}[1][]{\Ophypmsk^{\ifthenelse{\equal {#1} {}}{}{{#1}}}}
\newcommand{\rhshypmskd}[1][]{\rhshypmsk^{\ifthenelse{\equal {#1} {}}{}{{#1}}}}

\newcommand{\solhyphd}[1][]{\solhyph^{\ifthenelse{\equal {#1} {}}{}{{#1}}}}
\newcommand{\solhypsechd}[1][]{\solhypsech^{\ifthenelse{\equal {#1} {}}{}{{#1}}}}
\newcommand{\ptsolhypsechd}[1][]{\ptsolhypsech^{\ifthenelse{\equal {#1} {}}{}{{#1}}}}
\newcommand{\Ophyphd}[1][]{\Ophyph^{\ifthenelse{\equal {#1} {}}{}{{#1}}}}
\newcommand{\rhshyphd}[1][]{\rhshyph^{\ifthenelse{\equal {#1} {}}{}{{#1}}}}
\newcommand{\defectEuler}[1][]{\delta_\textnormal{BE}^{\ifthenelse{\equal {#1} {}}{}{{#1}}}}
\newcommand{\errmskd}[1][]{\errmsk^{\ifthenelse{\equal {#1} {}}{}{{#1}}}}
\newcommand{\errmsseckd}[1][]{\errmsseck^{\ifthenelse{\equal {#1} {}}{}{{#1}}}}
\newcommand{\pterrmsseckd}[1][]{\pterrmsseck^{\ifthenelse{\equal {#1} {}}{}{{#1}}}}
\newcommand{\errmskrhsd}[1][]{\errmskrhs^{\ifthenelse{\equal {#1} {}}{}{{#1}}}}

\newcommand{\OphypPGd}[1][]{\Ophyp_{k,\text{PG}}^{\ifthenelse{\equal {#1} {}}{}{{#1}}}}
\newcommand{\OphypsecPGd}[1][]{\Ophypsec_{k,\text{PG}}^{\ifthenelse{\equal {#1} {}}{}{{#1}}}}

\begin{document}

\author{Bernhard Maier\footnotemark[2]\and Barbara Verf\"urth\footnotemark[2]}
\date{}
\maketitle

\renewcommand{\thefootnote}{\fnsymbol{footnote}}
\footnotetext[2]{Institut für Angewandte und Numerische Mathematik, Karlsruher Institut f\"ur Technologie, Englerstr.~2, D-76131 Karlsruhe, \texttt{\{bernhard.maier, barbara.verfuerth\}@kit.edu}}
\renewcommand{\thefootnote}{\arabic{footnote}}

\begin{Abstract}
In this paper, we consider the classical wave equation with time-dependent, spatially multiscale coefficients.
We propose a fully discrete computational multiscale method in the spirit of the localized orthogonal decomposition in space with a backward Euler scheme in time.
We show optimal convergence rates in space and time beyond the assumptions of spatial periodicity or scale separation of the coefficients.
Further, we propose an adaptive update strategy for the time-dependent multiscale basis.
Numerical experiments illustrate the theoretical results and showcase the practicability of the adaptive update strategy.
\end{Abstract}

\begin{keywords}
wave equation,
numerical homogenization,
multiscale method,
time-dependent multiscale coefficients, 
a priori estimates
\end{keywords}

\begin{AMS}
35L05,
65M15,
65M60,
65N30
\end{AMS}

Time-modulated metamaterials \cite{AmmH21*,BalFP191} have recently received growing interest because of their astonishing physical properties such as the ability to break time-reversal symmetry~\cite{GuoDDN19,LiSZXC19}.
These can be modeled by the classical wave equation with space- and time-dependent coefficients.
Additionally, the metamaterials are characterized by fine spatial structures, such that these coefficients are in general rapidly varying on small spatial scales.
Besides time-modulated metamaterials, multiscale problems with time-dependent coefficients also occur in multiphysics simulations and for problems posed on evolving domains.
Furthermore, similar difficulties also arise for nonlinear wave-type problems, cf.~\cite{HocM21*,Mai21*,NikW19}.

Wave equations with time-dependent coefficients that are slowly varying in space were studied with finite element space discretizations and various time integration schemes in \cite{BalDS85,Bal86}.
However, in the metamaterial context with spatially multiscale coefficients, standard finite element methods need to resolve all scales leading to an enormous and often impractical computational effort.
Therefore, computational multiscale methods are suggested such as heterogeneous multiscale methods \cite{AbdG11,EngHR11}, multiscale finite element methods \cite{JiaE12,JiaEG10}, and rough polyharmonic splines \cite{OwhZ08}, to name a few.

In the present paper, we study wave equations with time-dependent multiscale coefficients, i.e., coefficients with a continuum of spatial scales that are only slowly evolving in time.
We propose a fully discrete computational multiscale method in the spirit of the Localized Orthogonal Decomposition (LOD) \cite{MalP14,HenP13,MalP20} with a backward Euler scheme for the time integration.
The LOD was successfully applied to time-harmonic wave propagation problems such as the Helmholtz equation, e.g., in \cite{GalP15,MaiV20*,Pet17,PetV20}, and Maxwell's equation \cite{GalHV17,HenP20}.
Wave equations with time-invariant multiscale coefficients were studied in \cite{AbdH17}, using the implicit Crank--Nicolson scheme for the integration in time.
In \cite{PetS17,MaiP19,GeeM21*} these results were extended to the explicit leapfrog scheme including mass lumping to further enhance the computational efficiency.
The presence of strong damping is considered in \cite{LjuMP21}.

Our main contribution is the first rigorous fully discrete a priori error analysis of the LOD for the wave equation with time-dependent coefficients.
We combine techniques for time-invariant multiscale \cite{AbdH17} and smooth time-dependent \cite{Bal86} coefficients in order to prove the expected order of convergence in space and time.
We strongly emphasize that merging these frameworks requires additional sophisticated concepts.
For instance, we prove exponential decay of time derivatives of the correction operators and rely on a special projection onto the multiscale space, which is not the usual Ritz projection.
Moreover, to enhance the computational efficiency of our scheme, we further propose an adaptive update strategy for the time-dependent multiscale basis using error indicators in the spirit of \cite{HelM19,HelKM20}.

Our paper is organized as follows. In~\Cref{sec:setting} we introduce the wave equation with time-dependent multiscale coefficients and our basic assumptions. \Cref{sec:lod} consists of two parts, where we first review the main concepts of the LOD and highlight the additional difficulties due to the time-dependent coefficients. Finally, we present the fully discrete numerical scheme and state our main error result. Subsequently, we provide a rigorous proof for this a priori error estimate in~\Cref{sec:errorAnalysis}. In~\Cref{sec:practicalAspects} we propose a Petrov--Galerkin variant of our scheme and introduce the adaptive update strategy. We conclude with numerical experiments that confirm our theoretical findings in~\Cref{sec:NumExp}.

\section{Setting} \label{sec:setting}
For $T > 0$ and a bounded polyhedral Lipschitz domain $\domain \subset \R^\spacedim$, $\spacedim \in \N$, we consider the wave equation with time-dependent multiscale coefficients
\begin{align} \label{eq:hypseceps}
  \left\{ ~
  \begin{aligned}
    \ptt \solhypseceps(t,x) &= \div \bigl( \aeps(t,x) \grad \solhypseceps(t,x) \bigr) + \rhshypsec(t,x), & t &\in [0,T], \, x \in \domain, \\
    \solhypseceps(t,x) &= 0\vphantom{\solhypsecinit}, & t &\in [0,T], \, x \in \partial \domain, \\
    \solhypseceps(0,x) &= \solhypsecinit, \qquad \pt \solhypseceps(0,x) = \ptsolhypsecinit, & x &\in \domain,
  \end{aligned}
  \right.
\end{align}
with given initial values $\solhypsecinit, \ptsolhypsecinit \in H_0^1(\domain)$ and right-hand side $\rhshypsec \in C^1([0,T],L^2(\domain))$. We assume that there are constants $\alc, \auc, \ptauc > 0$ which are in particular independent of $0 < \eps \ll 1$ such that the time-dependent multiscale parameter $\aeps \in C^1\bigl([0,T], L^\infty(\domain)\bigr)$  satisfies
\begin{align} \label{eq:aepsBounds}
  \alc &\leq \aeps(t,x) \leq \auc, & \abs{\pt \aeps(t,x)} &\leq \ptauc,
\end{align}
for all $t \in [0,T]$ and almost every $x \in \domain$. We emphasize that we do not assume $\aeps$ to be regular in space, since this would contradict the multiscale structure of $\aeps$.
We use the subscript $\eps$ to stress the multiscale nature of quantities, but assume neither periodicity nor scale separation.

\begin{remark} \label{rem:WellposednessEps}
  Note that for the sake of simplicity we only consider scalar-valued coefficients $\aeps$ here, but all arguments also extend to $\aeps$ being a symmetric matrix-valued coefficient. Further, we point out that wellposedness of \cref{eq:hypseceps} is shown in \cite{Gil72} under compatibility conditions on the initial data, i.e., on $\solhypsecinit$, $\ptsolhypsecinit$, and $\rhshypsec(0,\cdot)$. In particular, for vanishing initial data and $\rhshypsec \in W^{2,1}([0,T],L^2(\domain))$, there exists a unique solution $\solhypseceps$ of \cref{eq:hypseceps} satisfying
  \begin{equation} \label{eq:WellposednessEps}
    \solhypseceps \in C^3([0,T], L^2(\domain)) \cap C^2([0,T], H^1_0(\domain)).
  \end{equation}
\end{remark}

We introduce the Hilbert spaces
\begin{align}
  \Hhyp &= L^2(\domain), & \sprod{\Hhyp}{\varphi}{\psi} &= \sprod{L^2(\domain)}{\varphi}{\psi}, \\
  \Vhyp &= H_0^1(\domain), & \sprod{\Vhyp}{\varphi}{\psi} &= \sprod{L^2(\domain)}{\grad \varphi}{\grad \psi}.
\end{align}
Moreover, for $t \in [0,T]$ we use the time-dependent weighted inner product
\begin{align} \label{eq:defsprodVhypt}
  \sprod{\Vhypt,D}{\varphi}{\psi} &= \sprod{L^2(D)}{\aeps(t,\cdot) \grad \varphi}{\grad \psi}, & D &\subset \domain,
\end{align}
which is well defined due to \cref{eq:aepsBounds}. For $D = \domain$ we omit the subscript and simply write $\sprod{\Vhypt}{\cdot}{\cdot}$. Note that \cref{eq:aepsBounds} implies the norm equivalence
\begin{align} \label{eq:VVtnormequivalence}
  \alc \norm{\varphi}_{\Vhyp}^2 \leq \norm{\varphi}_{\Vhypt}^2 \leq \auc \norm{\varphi}_{\Vhyp}^2.
\end{align}
We further define the product spaces $\Xhyp = \Vhyp \times \Hhyp$ and $\Xhypt = \Vhypt \times \Hhyp$.

Based on these spaces, we study \cref{eq:hypseceps} with $\solhypeps = (\solhypseceps, \pt \solhypseceps)$ as the first-order Cauchy problem
\begin{align} \label{eq:hypeps}
  \left\{ ~
  \begin{aligned}
    \pt \solhypeps(t) &= \Ophypeps(t) \solhypeps(t) + \rhshyp(t), & t &\in [0,T], \\
    \solhypeps(0) &= \solhypinit,
  \end{aligned}
  \right.
\end{align}
with the initial value $\solhypinit = (\solhypsecinit, \ptsolhypsecinit)$, the right-hand side $\rhshyp = (0, \rhshypsec)$, and the time-dependent operator
\begin{align}
  \Ophypeps(t) &\colon D(\Ophypeps(t)) \to \Xhyp, & \Ophypeps(t) \begin{pmatrix} \varphi_1 \\ \varphi_2 \end{pmatrix} &= \begin{pmatrix} \varphi_2 \\ \div \bigl( \aeps(t,\cdot) \grad \varphi_1 \bigr) \end{pmatrix},
\end{align}
for $t \in [0,T]$. In particular, this implies for $\varphi = (\varphi_1, \varphi_2), \psi = (\psi_1, \psi_2)$
\begin{align}
  \sprod{\Xhypt}{\Ophypeps(t) \varphi}{\psi} &= \sprod{\Vhypt}{\varphi_2}{\psi_1} - \sprod{\Vhypt}{\varphi_1}{\psi_2}.
\end{align}

\section{Fully discrete localized orthogonal decomposition} \label{sec:lod}
At the end of this section, we propose a fully discrete numerical scheme for the wave equation with time-dependent multiscale coefficients \cref{eq:hypseceps}. To this end, we first introduce the spatially discrete setting in the spirit of the LOD in the first subsection. We then state the fully discrete scheme and our main result, which yields wellposedness as well as a rigorous error estimate.

\subsection{Localized orthogonal decomposition} \label{subsec:lod}
For time-invariant multiscale coefficients, the LOD is a well-established approach to improve the computational efficiency of numerical schemes. In particular, we refer to \cite{MalP20} for a detailed introduction to the topic. However, in our setting additional difficulties arise due to the dependency of the multiscale coefficient on time. Thus, we review the main ideas of the LOD in the following and highlight the additional difficulties. Overall, we mostly stick to the notation of \cite{MalP18}.

Let $\{\Triah\}_{h>0}$ and $\{\TriaH\}_{H>h}$ be two families of shape-regular and quasi-uniform triangulations of $\domain$, with $h$ and $H$ denoting the respective mesh width. Moreover, let $\FESpaceh$ and $\FESpaceH$ be the corresponding finite element spaces consisting of Lagrange elements of lowest order. In the following, we always assume that $\Triah$ is a refinement of $\TriaH$ such that $\FESpaceH \subset \FESpaceh$ is satisfied.

Further, since the construction of the scheme is built on quasi-local projections, we recursively define for $K \in \TriaH$ the patch of size $k \in \N_0$ by
\begin{align} \label{eq:notationPatches}
  \patch[0]{K} &= K, & \patch{K} &= \bigcup_{K \in \TriaH}\Bigl\{\overline{K} \cap \overline{\patch[k-1]{K}} \neq \emptyset\Bigr\}.
\end{align}

\paragraph{Fine finite element space}
In the following, we assume that the fine mesh width $h$ is chosen sufficiently small such that all oscillations of $\aeps$ are resolved. Thus, based on $\FESpaceh$ the standard finite element method is applicable. However, due to the smallness of $\eps$ and the associated high dimension of $\FESpaceh$, this is computationally at least very costly, if possible at all. Nevertheless, we introduce the scheme here, since this is used as the reference solution in the error estimates for our multiscale scheme. We emphasize that this reference solution is not needed for our scheme and is never computed in practice.

Based on the fine space $\FESpaceh$, we define the Hilbert spaces $\Hhyph = \FESpaceh$, $\Vhyph = \FESpaceh$, equipped with the inner product of $\Hhyp$ and $\Vhyp$, respectively. Additionally, we introduce the product space $\Xhyph = \Vhyph \times \Hhyph$.
We consider the discrete first-order Cauchy problem
\begin{align} \label{eq:Cauchyhypdisc}
  \left\{ ~
  \begin{aligned}
    \pt \solhyph(t) &= \Ophyph(t) \solhyph(t) + \rhshyph(t), & t &\in [0,T], \\
    \solhyph(0) &= \solhyphinit.
  \end{aligned}
  \right.
\end{align}
For $t \in [0,T]$ the time-dependent operator $\Ophyph(t) \colon \Xhyph \to \Xhyph$ is given by
\begin{align} \label{eq:defOphyph}
 \Ophyph(t) \begin{pmatrix} \varphi_1 \\ \varphi_2 \end{pmatrix} &= \begin{pmatrix} \varphi_2 \\ -\Ophypsech(t) \varphi_1 \end{pmatrix},
\end{align}
where the operator $\Ophypsech(t) \colon \Vhyph \to \Hhyph$ satisfies
\begin{align} \label{eq:defOphypsech}
  \sprod{\Hhyp}{\Ophypsech(t) \varphi_1}{\psi_2} &= \sprod{\Vhypt}{\varphi_1}{\psi_2}, & \varphi_1, \psi_2 &\in \Vhyph.
\end{align}
In general $\Ophypsech$ is not uniformly bounded with respect to $h$. Nevertheless, this implies
\begin{align} \label{eq:OphyphSkew}
 \sprod{\Xhypt}{\Ophyph(t) \varphi}{\psi} &= \sprod{\Vhypt}{\varphi_2}{\psi_1} - \sprod{\Vhypt}{\varphi_1}{\psi_2}, & \varphi, \psi &\in \Xhyph.
\end{align}
Moreover, the initial value $\solhyphinit \in \Xhyph$ and the right-hand side $\rhshypsech \colon [0,T] \to \Hhyph$ approximate $\solhypinit$ and $\rhshypsec$, respectively, and we set $\rhshyph = (0, \rhshypsech)$. Note that the wellposedness analysis in \cite{Gil72} also extends to the spatially discrete setting \cref{eq:Cauchyhypdisc}.

\paragraph{Coarse finite element space}
Correspondingly, for the coarse space $\FESpaceH$ we introduce the Hilbert spaces $\HhypH = \FESpaceH$, $\VhypH = \FESpaceH$, $\XhypH = \VhypH \times \HhypH$. Moreover, we introduce a projection $\InterpolH \in \Lin{\Vhyph,\VhypH}$, which satisfies for $K \in \TriaH$ the bounds
\begin{align} \label{eq:InterpolhHLocalBounds}
  \norm{\varphi - \InterpolH \varphi}_{\Hhyp,K} &\leq C_I H \norm{\varphi}_{\Vhyp,\patch[1]{K}}, & \norm{\varphi - \InterpolH \varphi}_{\Vhyp,K} &\leq C_I \norm{\varphi}_{\Vhyp,\patch[1]{K}}, & \varphi &\in \Vhyph,
\end{align}
with a constant $C_I > 0$ depending on the regularity of the mesh, but independent of $H$. In particular, this implies the global bounds
\begin{align} \label{eq:InterpolhHBounds}
  \norm{\varphi - \InterpolH \varphi}_{\Hhyp} &\lesssim C H \norm{\varphi}_{\Vhyp}, & \norm{\varphi - \InterpolH \varphi}_{\Vhyp} &\lesssim C \norm{\varphi}_{\Vhyp}, & \varphi &\in \Vhyph.
\end{align}
Moreover, we assume that $\InterpolH$ is not only stable in $\Vhyp$, but also in $\Hhyp$, i.e., we rely on
\begin{align} \label{eq:InterpolhHStableHhyp}
  \norm{\InterpolH \varphi}_{\Hhyp} &\lesssim C \norm{\varphi}_{\Hhyp}, & \varphi &\in \Hhyph.
\end{align}
Note that the quasi-interpolation operator introduced in \cite[Chap.~3.3]{MalP20}, which is also known as the Oswald operator, is a suitable choice for $\InterpolH$ satisfying our assumptions.

Based on $\InterpolH$ we decompose the fine space $\Vhyph = \Vhypf \oplus \VhypH$, where $\Vhypf = \ker \InterpolH$ and the dimension of $\VhypH$ is sufficiently small such that it allows for the construction of a computationally efficient scheme. However, we point out that this decomposition is not orthogonal with respect to the inner product of $\Vhyph$ and thus $\VhypH$ is not an optimal approximation space.

\paragraph{Time-dependent multiscale space}
We now enhance the coarse space $\FESpaceH$ to get an orthogonal decomposition of $\Vhyph$ with respect to the time-dependent inner product of $\Vhypt$. To this end, we introduce the time-dependent orthogonal projection $\ProjVhft \in \Lin{\Vhyph,\Vhypf}$ given by
\begin{align} \label{eq:defProjVhft}
  \sprod{\Vhypt}{\ProjVhft \varphi}{\psi} &= \sprod{\Vhypt}{\varphi}{\psi}, & t &\in [0,T], \, \varphi \in \Vhyph, \, \psi \in \Vhypf.
\end{align}
Furthermore, we make use of the corresponding mapping for time-dependent functions $\ProjVhf \colon C^1([0,T],\Vhyph) \to C^1([0,T],\Vhypf)$, which is well defined due to the regularity \cref{eq:aepsBounds} of $\aeps$ in time.
The time-dependent multiscale space is then defined by
\begin{align} \label{eq:defFESpacemst}
  \Vhypmst &= \bigl(\Id - \ProjVhft\bigr) \VhypH, & t &\in [0,T].
\end{align}
Note that for all $t \in [0,T]$ we have the orthogonal decomposition $\Vhyph = \Vhypmst \oplus \Vhypf$ with respect to $\sprod{\Vhypt}{\cdot}{\cdot}$.

We define for $t \in [0,T]$ the orthogonal projection $\ProjVhmst \in \Lin{\Vhyph,\Vhypmst}$ by
\begin{align} \label{eq:defProjVhmst}
  \sprod{\Vhypt}{\ProjVhmst \varphi}{\psi} &= \sprod{\Vhypt}{\varphi}{\psi}, & t &\in [0,T], \, \varphi \in \Vhyph, \, \psi \in \Vhypmst.
\end{align}
As before, we also introduce the corresponding mapping for time-dependent functions $\ProjVhms \colon C^1([0,T],\Vhyph) \to C^1([0,T],\Vhyph)$. In particular, we have $\ProjVhms = \Id - \ProjVhf$.


\paragraph{Localized time-dependent multiscale space}
Since $\ProjVhft$ is globally defined on $\domain$, the multiscale space $\Vhypmst$ in general consists of functions with global support and thus is not suited for the construction of an efficient numerical scheme. To circumvent this issue, we consider the element-wise contributions $\ProjVhfKt \colon \Vhyph \to \Vhypf$ given by
\begin{align} \label{eq:defProjVhfKt}
  \sprod{\Vhypt}{\ProjVhfKt \varphi}{\psi} &= \sprod{\Vhypt,K}{\varphi}{\psi}, & K &\in \TriaH, \, t \in [0,T], \, \varphi \in \Vhyph, \, \psi \in \Vhypf.
\end{align}
In particular, \cref{eq:defProjVhft} implies $\ProjVhf = \sum_{K \in \TriaH} \ProjVhfK$. As shown for instance in \cite[Chap.~4.1]{MalP20}, $\ProjVhfK \varphi$ decays exponentially. Moreover, in \Cref{lem:ProjfProjfk} we show that this also holds for the time derivatives $\pt \bigl( \ProjVhfK \varphi\bigr)$. This motivates the localization of $\ProjVhf$, $\ProjVhms$, and $\Vhypmst$.

Based on the notation introduced in \cref{eq:notationPatches}, the localized element-wise contributions $\ProjVhfkKt \colon \Vhyph \to \Vhypf\bigl(\patch{K}\bigr)$ are given by
\begin{align} \label{eq:defProjVhfkKt}
  \sprod{\Vhypt,\patch{K}}{\ProjVhfkKt \varphi}{\psi} &= \sprod{\Vhypt,K}{\varphi}{\psi}, & t &\in [0,T], \, \varphi \in \Vhyph, \, \psi \in \Vhypf\bigl(\patch{K}\bigr),
\end{align}
using the localized fine space $\Vhypf\bigl(\patch{K}\bigr) = \bigl\{ \varphi \in \Vhypf \mid \operatorname{supp} \varphi \subset \patch{K}\bigr\}$.

Recombination yields the localized projection $\ProjVhfk = \sum_{K \in \TriaH} \ProjVhfkK$. Again, we also use the mapping for time-dependent functions $\ProjVhfk \colon C^1([0,T],\Vhyph) \to C^1([0,T],\Vhypf)$. Moreover, similarly to \cref{eq:defFESpacemst} and \cref{eq:defProjVhmst} we define the localized time-dependent multiscale space
\begin{align} \label{eq:defFESpacemskt}
  \Vhypmskt &= \bigl(\Id - \ProjVhfkt\bigr) \VhypH, & t &\in [0,T],
\end{align}
and $\ProjVhmskt \in \Lin{\Vhyph,\Vhypmskt}$ as well as $\ProjVhmsk \colon C^1([0,T],\Vhyph) \to C^1([0,T],\Vhyph)$ via
\begin{align} \label{eq:defProjVhmskt}
  \sprod{\Vhypt}{\ProjVhmskt \varphi}{\psi} &= \sprod{\Vhypt}{\varphi}{\psi}, & t &\in [0,T], \, \varphi \in \Vhyph, \, \psi \in \Vhypmskt.
\end{align}
Additionally, we introduce $\ProjVhmsktildet \in \Lin{\Vhyph,\Vhypmskt}$ and correspondingly for time-dependent functions $\ProjVhmsktilde \colon C^1([0,T],\Vhyph) \to C^1([0,T],\Vhyph)$ by
\begin{align} \label{eq:defProjVhmsktildet}
  \ProjVhmsktilde &= \bigl( \Id - \ProjVhfk \bigr) \InterpolH.
\end{align}
In contrast to the (Ritz) projection $\ProjVhmsk$, the projection $\ProjVhmsktilde$  is not orthogonal to $\Vhypmskt$ with respect to the inner product of $\Vhypt$. We emphasize that, nevertheless, the structure of $\ProjVhmsktilde$ is very similar to the ideal multiscale projection $\ProjVhms = \Id - \ProjVhf$. This is a key aspect of our error analysis.

Finally, we introduce the space $\Hhypmskt$ consisting of the same functions as $\Vhypmskt$, but equipped with the inner product of $\Hhyp$, as well as $\Xhypmskt = \Vhypmskt \times \Hhypmskt$. The corresponding time-dependent multiscale projections $\ProjHhmskt \in \Lin{\Hhyph,\Hhypmskt}$ and $\ProjHhmsk \colon C^1([0,T],\Hhyph) \to C^1([0,T],\Hhyph)$ satisfy
\begin{align} \label{eq:defProjHhmskt}
  \sprod{\Hhyp}{\ProjHhmskt \varphi}{\psi} &= \sprod{\Hhyp}{\varphi}{\psi}, & t &\in [0,T], \, \varphi \in \Hhyph, \, \psi \in \Hhypmskt.
\end{align}

\begin{remark} \label{eq:timeDependentMultiscaleSpaces}
	The time-dependency of $\aeps$ requires the multiscale spaces to be time-dependent as well, which largely influences the fully discrete scheme and its error analysis below.
	In the special case where $\aeps(t,x)=a_{1,\eps}(x)a_2(t)$, the localized element projections $\ProjVhfkKt$ in \cref{eq:defProjVhfkKt} can be equivalently formulated with $a_{1,\eps}$ only, so that they become time-independent. Hence, in this special case, we can work with time-independent multiscale spaces and the following error analysis can be simplified. 
	We illustrate this fact also in the numerical experiments.
\end{remark}

\subsection{Fully discrete scheme} \label{subsec:fullyDiscreteScheme}
Following the method-of-lines approach, we now present the fully discrete localized orthogonal decomposition method for the wave equation with time-dependent multiscale coefficients \cref{eq:hypeps} based on the backward Euler scheme for the discretization in time. To this end, let $\tau > 0$ denote the fixed time step, $N \in \N$ with $N \tau \leq T$, and $\timed{m} = m \tau$ for all $m \leq N$. We introduce for $m = 0,\dots,N$ the short notation $\Xhypd[m] = \Xhypt[\timed{m}]$, $\Xhypmskd[m] = \Xhypmskt[\timed{m}]$, and the projections $\ProjXhmskd[m], \RefXhmsktilded[m] \in \Lin{\Xhyph,\Xhypmskd[m]}$ given by
\begin{align} \label{eq:notationProjs}
  \ProjXhmskd[m] &= \begin{pmatrix} \ProjVhmskt[\timed{m}] \\ \ProjHhmskt[\timed{m}] \end{pmatrix}, & \RefXhmsktilded[m] &= \begin{pmatrix} \ProjVhmsktildet[\timed{m}] \\ \ProjVhmsktildet[\timed{m}] \end{pmatrix}.
\end{align}
Moreover, we define
\begin{align} \label{eq:notationDiscreteOps}
 \Ophypmskd[m] &= \ProjXhmskd[m] \Ophyph(\timed{m}), & \rhshypmskd[m] &= \ProjXhmskd[m] \rhshyph(\timed{m}), & \solhypmskinit &= \ProjXhmskd[0] \solhyphinit.
\end{align}
The fully discrete scheme then reads
\begin{align} \label{eq:backwardEulerLOD}
  \solhypmskd[n+1] &= \ProjXhmskd[n+1] \solhypmskd[n] + \tau \Ophypmskd[n+1] \solhypmskd[n+1] + \tau \rhshypmskd[n+1], & n &=0,\dots,N-1.
\end{align}
Note that we apply the projection $\ProjXhmskd[n+1]$ to the previous approximation $\solhypmskd[n]$ in order to ensure $\solhypmskd[n+1] \in \Xhypmskd[n+1]$. Thus, the right-hand side of \cref{eq:backwardEulerLOD} is well defined.

We now state our main result, which yields wellposedness and a rigorous error estimate for the fully discrete scheme.
\begin{theorem} \label{thm:backwardEulerError}
  Let $\solhyph = (\solhypsech, \pt \solhypsech)$ be the solution of \cref{eq:Cauchyhypdisc} with
  \begin{align} \label{eq:regularityforEulererrornonautonomouswave}
    \solhypsech &\in C^3([0,T], \Hhyph) \cap C^2([0,T], \Vhyph)
  \end{align}
bounded independent of $\eps$.
  Moreover, let $\patchsize \geq \bigl(1 + \frac{\spacedim}{2}\bigr)\abs{\log H}/\abs{\log \mu}$. Then, the approximations $\solhypmskd[n] = (\solhypsecmskd[n], \ptsolhypsecmskd[n])$ obtained by the fully discrete localized orthogonal decomposition method \cref{eq:backwardEulerLOD} for the wave equation with time-dependent multiscale coefficients \cref{eq:hypeps} satisfy for $n = 1,\dots,N$ the error estimate
  \begin{align} \label{eq:Eulererrorestimate}
    \norm{\solhypsecmskd[n] - \solhypsech(\timed{n})}_{\Vhyp} + \norm{\ptsolhypsecmskd[n] - \pt \solhypsech(\timed{n})}_{\Hhyp} &\leq \mathcal{C}_{\solhypsec,\rhshypsec} \expp{\mathcal{C}_\acoeff \timed{n}} \bigl(\tau + H\bigr),
  \end{align}
  where the constants $\mathcal{C}_{\solhypsec,\rhshypsec}, \mathcal{C}_\acoeff > 0$ depend on $\alc$, $\auc$, $\ptauc$, and the mesh regularity, but not on the variations of $\aeps$. Further, $\mathcal{C}_{\solhypsec,\rhshypsec}$ depends on \cref{eq:regularityforEulererrornonautonomouswave}, but not on $\eps$. The constant $0 < \mu < 1$ is given by \Cref{lem:ProjfProjfk}.
\end{theorem}
The proof of the main result is given at the end of the next section. We conclude this section with the following remarks.
\begin{remark}
  As stated in \Cref{rem:WellposednessEps}, wellposedness of wave equations with time-dependent coefficients of the form \cref{eq:hypeps} and \cref{eq:Cauchyhypdisc} is studied in \cite{Gil72}. In particular, the author derives compatibility conditions on the initial data and the right-hand side such that the regularity assumption \cref{eq:regularityforEulererrornonautonomouswave} on $\solhypsech$ holds. 
  In order to have $\eps$-independent a priori bounds on $u_h$ in the associated norms, it is additionally required that the initial values and the right-hand side satisfy $\eps$-independent bounds in appropriate norms, cf.~\cite{AbdH17} for a detailed discussion in the case of time-independent $\aeps$.
  For instance, it is shown in \cite{Gil72} that the assumption of well-prepared data is satisfied for vanishing initial data and $\rhshypsech \in W^{2,1}([0,T],\Hhyph)$.
\end{remark}
\begin{remark}
  Similar arguments as in the proof of \Cref{thm:backwardEulerError} also yield an error estimate for the spatially discrete problem without the discretization in time. However, since wellposedness thereof is much more involved, we refrain from presenting the details here.
\end{remark}
\begin{remark} \label{rem:implicitMidpoint}
  We emphasize that our analysis for the backward Euler scheme offers a good starting point for the derivation of higher-order schemes. For instance, correspondingly to \cref{eq:backwardEulerLOD} we propose the following fully discrete scheme for the wave equation with time-dependent multiscale coefficients \cref{eq:hypeps} based on the implicit midpoint rule, which is given by
  \begin{align} \label{eq:implicitMidpointLOD}
    \begin{aligned}
      \solhypmskd[n+1/2] &= \ProjXhmskd[n+1/2] \solhypmskd[n] + \tfrac{\tau}{2} \Ophypmskd[n+1/2] \solhypmskd[n+1/2] + \tfrac{\tau}{2} \rhshypmskd[n+1/2], \\
      \solhypmskd[n+1] &= 2 \solhypmskd[n+1/2] - \solhypmskd[n],
    \end{aligned}
    && n &=0,\dots,N-1.
  \end{align}
  The discrete operators are defined as in \cref{eq:notationProjs} and \cref{eq:notationDiscreteOps} with $\timed{m+1/2} = \timed{m} + \frac{\tau}{2}$. In particular, \cref{eq:implicitMidpointLOD} implies $\solhypmskd[n+1/2] \in \Xhypmskd[n+1/2]$, which is crucial for the right-hand side to be well defined, but not $\solhypmskd[n+1] \in \Xhypmskd[n+1]$. We point out that in our numerical experiments in \Cref{sec:NumExp} this scheme is second-order convergent. However, since the corresponding analysis is much more involved, we focus in the present paper on the analysis of the backward Euler scheme.
\end{remark}

\section{Error analysis} \label{sec:errorAnalysis}
The aim of this section is to provide a rigorous proof for our main result \Cref{thm:backwardEulerError}. To this end, in the first subsection we investigate the approximation properties of the projection operators. Based on these estimates, we finally conclude an error estimate for the fully discrete scheme in the second subsection.

\subsection{Approximation properties of projections}
In the following, we study the projections introduced in \Cref{subsec:lod}. In particular, this includes the derivation of estimates for the corresponding time derivatives.

To begin with, the definition \cref{eq:defProjVhmsktildet} of $\ProjVhmsktilde$ and the identity $\ProjVhms \bigl(\Id - \InterpolH\bigr) = 0$ in $\Vhyph$ imply
\begin{align} \label{eq:IdProjVhmsktildetIdentity}
  \ProjVhmsktilde - \Id = \ProjVhms - \Id + \bigl( \ProjVhf - \ProjVhfk \bigr) \InterpolH.
\end{align}
Thus, in the following lemmas we study the ideal multiscale projection $\ProjVhms$ and the perturbation error due to the localization $\ProjVhfk$ of the finescale projection $\ProjVhf$. Finally, we conclude bounds for the non-orthogonal multiscale projection $\ProjVhmsktilde$.

We start with the ideal multiscale projection $\ProjVhms$ introduced in \cref{eq:defProjVhmst}. 
\begin{lemma} \label{lem:IdProjmsVh}
  Let $t \in [0,T]$ and $\varphi \in C^1([0,T],\Vhyph)$ with $\Ophypsech(t) \varphi(t) \in \Hhyph$. Then, we have
  \begin{align}
    \norm{\bigl(\Id - \ProjVhmst\bigr) \varphi(t)}_{\Vhyp} &\lesssim H \norm{\Ophypsech(t) \varphi(t)}_{\Hhyp}, \label{eq:IdProjmsVh} \\
    \norm{\pt\bigl((\Id - \ProjVhms) \varphi\bigr)(t)}_{\Vhyp} &\lesssim H \Bigl(\norm{\Ophypsech(t) \varphi(t)}_{\Hhyp} + \norm{\pt \bigl( \Ophypsech(t) \varphi(t) \bigr)}_{\Hhyp}\Bigr), \label{eq:ptIdProjmsVh}
  \end{align}
  where the hidden constant might depend on $\alc$, $\auc$, $\ptauc$, and the mesh regularity, but not on the spatial variations of $\aeps$.
\end{lemma}
\begin{proof}
  The estimate \cref{eq:IdProjmsVh} follows directly from $\ProjVhf = \Id - \ProjVhms$, the norm equivalence \cref{eq:VVtnormequivalence}, and \cref{eq:InterpolhHBounds}. To prove \cref{eq:ptIdProjmsVh}, we first obtain due to $\pt\bigl(\ProjVhf \varphi\bigr) \in C^1([0,T],\Vhypf)$
  \begin{align}
    \norm{\pt\bigl((\Id - \ProjVhms) \varphi\bigr)(t)}_{\Vhypt} &= \sup_{\stackrel{\psi \in \Vhypf}{\norm{\psi}_{\Vhypt}=1}} \sprod{\Hhyp}{\aeps(t) \grad \pt\bigl(\ProjVhf \varphi\bigr)(t)}{\grad \psi}.
  \end{align}
  The product rule together with \cref{eq:defProjVhft} implies
  \begin{align} \label{eq:ptProjVhft}
    \begin{aligned}
      \MoveEqLeft \sprod{\Hhyp}{\aeps(t) \grad \pt\bigl(\ProjVhf \varphi\bigr)(t)}{\grad \psi} \\
      &= \pt \sprod{\Hhyp}{\aeps(t) \grad \varphi(t)}{\grad \psi} - \sprod{\Hhyp}{\pt \aeps(t) \grad\bigl(\ProjVhft \varphi(t)\bigr)}{\grad \psi}.
    \end{aligned}
  \end{align}
  For the first term, we derive with the definition \cref{eq:defOphypsech} of $\Ophypsech$, $\psi \in \ker \InterpolH$, and \cref{eq:InterpolhHBounds} the bound
  \begin{align}
    \abs{\pt \sprod{\Hhyp}{\aeps(t) \grad \varphi(t)}{\grad \psi}} &\leq \norm{\pt \bigl( \Ophypsech(t) \varphi(t) \bigr)}_{\Hhyp} \norm{(\Id - \InterpolH) \psi}_{\Hhyp} \\
    &\lesssim H \norm{\pt \bigl( \Ophypsech(t) \varphi(t) \bigr)}_{\Hhyp} \norm{\psi}_{\Vhyp}.
  \end{align}
  For the second term, the Cauchy--Schwarz inequality yields
  \begin{align}
    \sprod{\Hhyp}{\pt \aeps(t) \grad\bigl(\ProjVhft \varphi(t)\bigr)}{\grad \psi} &\leq \norm{\pt \aeps(t)}_{L^\infty(\domain)} \norm{\bigl( \Id - \ProjVhmst \bigr) \varphi(t)}_{\Vhyp} \norm{\psi}_{\Vhyp}.
  \end{align}
  Finally, \cref{eq:IdProjmsVh} and the norm equivalence \cref{eq:VVtnormequivalence} conclude the proof.
\end{proof}

Although $\norm{\Ophypsech(t) \varphi}_{\Hhyp}$ is in general not uniformly bounded in $h$, we point out that the finite element solution $\solhyph = (\solhypsech, \pt \solhypsech) \in C^1([0,T],\Xhyph)$ of \cref{eq:Cauchyhypdisc} satisfies
\begin{align} \label{eq:boundOphypsechsolhypsech}
  \norm{\Ophypsech(t) \solhypsech(t)}_{\Hhyp} = \norm{\ptt \solhypsech(t) - \rhshypsech(t)}_{\Hhyp}.
\end{align}
Thus, for $\varphi = \solhypsech$ and $\rhshyph \in C([0,T],\Xhyph)$ the right-hand side of \cref{eq:IdProjmsVh} is uniformly bounded in $h$. Moreover, taking the derivative of \cref{eq:Cauchyhypdisc} with respect to time, we obtain
\begin{align} \label{eq:boundptOphypsechsolhypsech}
  \norm{\pt \bigl(\Ophypsech(t) \solhypsech(t)\bigr)}_{\Hhyp} &= \norm{\pttt \solhypsech(t) - \pt \rhshypsech(t)}_{\Hhyp}.
\end{align}
Thus, for $\varphi = \solhypsech$ the right-hand side of \cref{eq:ptIdProjmsVh} is also uniformly bounded in $h$ if we have $\solhyph \in C^2([0,T],\Xhyph)$ and $\rhshyph \in C^1([0,T],\Xhyph)$.

However, note that for time-dependent multiscale coefficients the same trick is not feasible for $\varphi = \pt \solhypsech$, since the product rule yields
\begin{align}
  \norm{\Ophypsech(t) \pt \solhypsech(t)}_{\Hhyp} &= \norm{\pttt \solhypsech(t) - \pt \rhshypsech(t) - \pt \bigl(\Ophypsech(t)\bigr) \solhypsech(t)}_{\Hhyp},
\end{align}
but $\norm{\pt \bigl(\Ophypsech(t)\bigr) \solhypsech(t)}_{\Hhyp}$ is in general not uniformly bounded in $h$. Thus, we provide alternative bounds in the following lemma using only the energy norm on the right-hand side.
\begin{lemma} \label{lem:IdProjmsHh}
  Let $t \in [0,T]$ and $\varphi \in C^1([0,T],\Vhyph)$. Then, we have
  \begin{align}
    \norm{\bigl(\Id - \ProjVhmst\bigr) \varphi(t)}_{\Hhyp} &\lesssim H \norm{\varphi(t)}_{\Vhyp}, \label{eq:IdProjmsHh} \\
    \norm{\pt\bigl((\Id - \ProjVhms) \varphi\bigr)(t)}_{\Hhyp} &\lesssim H \bigl(\norm{\varphi(t)}_{\Vhyp} + \norm{\pt \varphi(t)}_{\Vhyp}\bigr), \label{eq:ptIdProjmsHh}
  \end{align}
  where the hidden constant might depend on $\alc$, $\auc$, $\ptauc$, and the mesh regularity, but not on the spatial variations of $\aeps$.
\end{lemma}
\begin{proof}
  The bound \cref{eq:IdProjmsHh} follows from $\ProjVhf = \Id - \ProjVhms$, \cref{eq:InterpolhHBounds}, and \cref{eq:defProjVhft}. Furthermore, we obtain with $\ProjVhf \colon C^1([0,T],\Vhyph) \to C^1([0,T],\Vhypf)$ and \cref{eq:InterpolhHBounds}
  \begin{align}
    \norm{\pt\bigl(\ProjVhf \varphi\bigr)(t)}_{\Hhyp}
    &\lesssim H \norm{\pt\bigl(\ProjVhf \varphi\bigr)(t)}_{\Vhyp}.
  \end{align}
  Since \cref{eq:VVtnormequivalence}, \cref{eq:defProjVhft}, and \cref{eq:ptProjVhft} imply
  \begin{align}
    \norm{\pt\bigl(\ProjVhf \varphi\bigr)(t)}_{\Vhyp} &\lesssim \bigl(\norm{\varphi(t)}_{\Vhyp} + \norm{\pt \varphi(t)}_{\Vhyp}\bigr),
  \end{align}
  we also have \cref{eq:ptIdProjmsHh}.
\end{proof}

So far we only studied the ideal multiscale projection $\ProjVhms$. In order to obtain similar bounds for the localized projection $\ProjVhmsktilde$, we first study the error introduced by the localization.
\begin{proposition} \label{lem:ProjfProjfk}
  Let $t \in [0,T]$ and $\varphi \in C^1([0,T],\Vhyph)$. There is a constant $0 < \mu < 1$ depending on $\frac{\alc}{\auc}$ such that
  \begin{align}
    \norm{\bigl(\ProjVhft - \ProjVhfkt\bigr) \varphi(t)}_{\Vhyp} &\lesssim \patchsize^{\spacedim/2} \mu^\patchsize \norm{\varphi(t)}_{\Vhyp}, \label{eq:ProjfProjfk} \\
    \norm{\pt \bigl((\ProjVhf - \ProjVhfk) \varphi\bigr)(t)}_{\Vhyp} &\lesssim H^{-\spacedim/2} \mu^\patchsize \bigl(\norm{\varphi(t)}_{\Vhyp} + \norm{\pt \varphi(t)}_{\Vhyp}\bigr). \label{eq:ptProjfProjfk}
  \end{align}
  The hidden constant might depend on $\alc$, $\auc$, $\ptauc$, and the mesh regularity, but not on the variations of $\aeps$.
\end{proposition}
\begin{proof}
  The estimate \cref{eq:ProjfProjfk} was shown for instance in \cite[Lem.~3.6]{HenM14}. Thus, we focus on the proof of \cref{eq:ptProjfProjfk} in the following, which consists of three parts. In the first two parts, we show the local estimates
  \begin{align}
    \norm{\pt\bigl(\ProjVhfK \varphi\bigr)(t)}_{\Vhyp,\domain\setminus\patch{K}} &\lesssim \mu^k \bigl(\norm{\varphi(t)}_{\Vhyp,K} + \norm{\pt \varphi(t)}_{\Vhyp,K}\bigr), \label{eq:exponentialDecayProjVhfK} \\
    \norm{\pt\bigl((\ProjVhfK - \ProjVhfkK)\varphi\bigr)(t)}_{\Vhyp} &\lesssim \mu^k \bigl(\norm{\varphi(t)}_{\Vhyp,K} + \norm{\pt \varphi(t)}_{\Vhyp,K}\bigr). \label{eq:exponentialDecayProjVhfKProjVhfkK}
  \end{align}
  In the last part, we then conclude the global estimate \cref{eq:ptProjfProjfk}. Note that in the remainder of the proof we omit the time-dependency of functions for the sake of readability. Moreover, we refrain from specifying the element $K \in \TriaH$ for each patch $\patch{K}$ whenever this is clear from the context.
  
  \vspace{2ex}
  
  \textit{Step 1: Proof of \cref{eq:exponentialDecayProjVhfK}:} A corresponding estimate for constant-in-time coefficients is shown in \cite[Thm.~4.1]{MalP20}. In the following, we extend this result to time-dependent multiscale coefficients.
  
  For $k \geq 4$ and $K \in \TriaH$ fixed we introduce the cut-off function $\eta_1 \in \VhypH$ with
  \begin{align}
    \eta_1 &= 0 \quad\textnormal{in } \patch[k-3]{K}, & \eta_1 &= 1 \quad\textnormal{in } \domain\setminus\patch[k-2]{K}.
  \end{align}
  Thus, we obtain for $\widetilde{\varphi} = \ProjVhfK \varphi$ the estimate
  \begin{align}
    \norm{\pt \widetilde{\varphi}}_{\Vhypt,\domain\setminus\patch{}}^2 &\leq \sprod{\Vhypt}{\pt \widetilde{\varphi}}{(\Id - \InterpolH)(\eta_1 \pt \widetilde{\varphi})} + \sprod{\Vhypt}{\pt \widetilde{\varphi}}{\InterpolH(\eta_1 \pt \widetilde{\varphi})} \\
    &\hphantom{\leq{}}- \sprod{\Hhyp}{\aeps \grad \pt \widetilde{\varphi}}{(\nabla \eta_1) \pt \widetilde{\varphi}}.
  \end{align}
  For the first term, we derive from \cref{eq:defProjVhfKt} for $\psi \in \Vhypf$
  \begin{align} \label{eq:ptProjVhfK}
    \begin{aligned}
      \sprod{\Vhypt}{\pt\bigl(\ProjVhfK \varphi\bigr)}{\psi} &= \sprod{\Vhypt,K}{\pt \varphi}{\psi} + \sprod{\Hhyp,K}{\pt \aeps \grad \varphi}{\grad \psi} \\
      &\hphantom{={}}- \sprod{\Hhyp}{\pt \aeps \grad \bigl(\ProjVhfK \varphi\bigr)}{\grad \psi}.
    \end{aligned}
  \end{align}
  In particular, this yields for $\psi = (\Id - \InterpolH)(\eta_1 \pt \widetilde{\varphi}) \in \Vhypf\bigl(\domain\setminus\patch[k-4]{K}\bigr)$
  \begin{align}
    \sprod{\Vhypt}{\pt \widetilde{\varphi}}{(\Id - \InterpolH)(\eta_1 \pt \widetilde{\varphi})} &= -\sprod{\Hhyp}{\pt \aeps \grad \widetilde{\varphi}}{\grad \psi} \\
    &\lesssim \norm{\widetilde{\varphi}}_{\Vhyp,\domain\setminus\patch[k-4]{}} \norm{\pt \widetilde{\varphi}}_{\Vhyp,\domain\setminus\patch[k-4]{}},
  \end{align}
  where we used \cref{eq:InterpolhHBounds} in the last step. Since the other two terms can be bounded as in \cite[Thm.~4.1]{MalP20}, we obtain with constants $C_1, C_2 > 0$ 
  \begin{align}
    \norm{\pt \widetilde{\varphi}}_{\Vhypt,\domain\setminus\patch{}}^2 &\leq C_1 \norm{\pt \widetilde{\varphi}}_{\Vhypt,\patch[]{}\setminus\patch[k-4]{}}^2 + C_2 \norm{\widetilde{\varphi}}_{\Vhyp,\domain\setminus\patch[k-4]{}} \norm{\pt \widetilde{\varphi}}_{\Vhyp,\domain\setminus\patch[k-4]{}}.
    \end{align}
	The identity $\patch[]{}\setminus\patch[k-4]{}=(\domain\setminus\patch[k-4]{})\setminus (\domain\setminus \patch[]{})$ and Young's inequality for $\varepsilon > 0$ yield
	\begin{align}
	\norm{\pt \widetilde{\varphi}}_{\Vhypt,\domain\setminus\patch{}}^2
    &\leq \bigl(\tfrac{C_1}{1+C_1} + \tfrac{\varepsilon}{2}\bigr) \norm{\pt \widetilde{\varphi}}_{\Vhypt,\domain\setminus\patch[k-4]{}}^2 + \tfrac{1}{2 \varepsilon} \bigl(\tfrac{C_2}{1+C_1}\bigr)^2 \norm{\widetilde{\varphi}}_{\Vhyp,\domain\setminus\patch[k-4]{}}^2.
  \end{align}
  Due to $\tfrac{C_1}{1+C_1} < 1$ there is $\varepsilon$ sufficiently small such that using the previous argument repeatedly yields
  \begin{align}
    \norm{\pt \widetilde{\varphi}}_{\Vhypt,\domain\setminus\patch{}}^2 &\leq \mu^{\lfloor k/4 \rfloor} \norm{\pt \widetilde{\varphi}}_{\Vhypt,\domain}^2 + C_3 \sum_{i=1}^{\lfloor k/4 \rfloor} \mu^{i-1} \norm{\widetilde{\varphi}}_{\Vhyp,\domain\setminus\patch[k-4i]{}}^2,
  \end{align}
  for some $0 < \mu < 1$ and $C_3 > 0$. Thus, we finally obtain with
  \begin{align}
    \norm{\pt \widetilde{\varphi}}_{\Vhyp} \lesssim \norm{\varphi}_{\Vhyp, K} + \norm{\pt \varphi}_{\Vhyp, K}
  \end{align}
  and the exponential decay of $\widetilde{\varphi} = \ProjVhfK \varphi$ from \cite[Thm.~4.1]{MalP20} the estimate \cref{eq:exponentialDecayProjVhfK}, which remains valid for $k < 4$.
  
  \vspace{2ex}
  
  \textit{Step 2: Proof of \cref{eq:exponentialDecayProjVhfKProjVhfkK}:} For the second part, \cref{eq:defProjVhfkKt} implies for $\psi_k \in \Vhypf\bigl(\patch{K}\bigr)$
  \begin{align}
    \begin{aligned}
      \sprod{\Vhypt}{\pt\bigl(\ProjVhfkK \varphi\bigr)}{\psi_k} &= \sprod{\Vhypt,K}{\pt \varphi}{\psi_k} + \sprod{\Hhyp,K}{\pt \aeps \grad \varphi}{\grad \psi_k} \\
      &\hphantom{={}}- \sprod{\Hhyp}{\pt \aeps \grad \bigl(\ProjVhfkK \varphi\bigr)}{\grad \psi_k}.
    \end{aligned} \label{eq:ptProjVhfkK}
  \end{align}
  Together with \cref{eq:ptProjVhfK}, this yields for $w = (\ProjVhfK - \ProjVhfkK)\varphi$ and $\psi_k \in \Vhypf\bigl(\patch{K}\bigr)$ due to $\widetilde{\psi}_k = \pt\bigl(\ProjVhfkK \varphi\bigr) - \psi_k \in \Vhypf\bigl(\patch{K}\bigr)$ the estimate
  \begin{align}
    \norm{\pt w}_{\Vhypt}^2 &= \sprod{\Vhypt}{\pt w}{\pt w + \widetilde{\psi}_k} + \sprod{\Hhyp}{\pt \aeps \grad w}{\grad \widetilde{\psi}_k} \\
    &= \sprod{\Vhypt}{\pt w}{\pt \widetilde{\varphi} - \psi_k} + \sprod{\Hhyp}{\pt \aeps \grad w}{\grad \bigl(\pt \widetilde{\varphi} - \psi_k\bigr)} \\
    &\hphantom{={}}- \sprod{\Hhyp}{\pt \aeps \grad w}{\grad \bigl(\pt w\bigr)} \\
    &\lesssim \bigl(\norm{\pt w}_{\Vhyp} + \norm{w}_{\Vhyp}\bigr) \norm{\pt \widetilde{\varphi} - \psi_k}_{\Vhyp} + \norm{w}_{\Vhyp} \norm{\pt w}_{\Vhyp},
  \end{align}
  where we used the Cauchy--Schwarz inequality together with the norm equivalence \cref{eq:VVtnormequivalence} and the regularity \cref{eq:aepsBounds} of $\aeps$ in the last step. Since $\psi_k$ is arbitrary, we have shown
  \begin{align}
    \norm{\pt w}_{\Vhyp} &\lesssim \norm{w}_{\Vhyp} + \inf_{\psi_k \in \Vhypf(\patch{})} \norm{\pt \widetilde{\varphi} - \psi_k}_{\Vhyp}.
  \end{align}
  Using the estimate for $w = (\ProjVhfK - \ProjVhfkK)\varphi$ from \cite[Cor.~4.2]{MalP20}, the first term is bounded by the right-hand side of \cref{eq:exponentialDecayProjVhfKProjVhfkK}. For the second term, without loss of generality let $k \geq 3$. We choose $\psi_k = (\Id - \InterpolH)(\eta_2 \pt \widetilde{\varphi})$ with the cut-off function $\eta_2 \in \VhypH$ with
  \begin{align}
    \eta_2 &= 1 \quad\textnormal{in } \patch[k-2]{K}, & \eta_2 &= 1 \quad\textnormal{in } \domain\setminus\patch[k-1]{K}.
  \end{align}
  The estimate \cref{eq:exponentialDecayProjVhfKProjVhfkK} then follows with the same arguments as in the proof of \cite[Cor.~4.2]{MalP20}.
  
  \vspace{2ex}
  
  \textit{Step 3: Proof of \cref{eq:ptProjfProjfk}:} To obtain the global bound \cref{eq:ptProjfProjfk}, we use \cref{eq:exponentialDecayProjVhfKProjVhfkK} and the triangle inequality to get
  \begin{align}
    \norm{\pt \bigl((\ProjVhf - \ProjVhfk)\varphi\bigr)}_{\Vhyp} &\leq \sum_{K \in \TriaH} \norm{\pt \bigl((\ProjVhfK - \ProjVhfkK)\varphi\bigr)}_{\Vhyp} \\
    &\lesssim \mu^k \sum_{K \in \TriaH} \bigl(\norm{\varphi}_{\Vhyp,K} + \norm{\pt \varphi}_{\Vhyp,K}\bigr).
  \end{align}
  Due to the mesh regularity, this yields \cref{eq:exponentialDecayProjVhfKProjVhfkK}.
\end{proof}
\begin{remark}
  Compared to \cref{eq:ProjfProjfk} the estimate \cref{eq:ptProjfProjfk} for the time derivative seems to be  sub-optimal. However, we point out that due to the dominant exponential decay, this only mildly affects the patch size $\patchsize$; e.g., the choice $\patchsize \geq \bigl(1 + \frac{\spacedim}{2}\bigr)\abs{\log H}/\abs{\log \mu}$ in \Cref{thm:backwardEulerError} is sufficient to ensure the optimal convergence rates of the fully discrete scheme. Nevertheless, in future research it might be possible to improve the estimate \cref{eq:ptProjfProjfk}.
\end{remark}

Due to \cref{eq:IdProjVhmsktildetIdentity}, we directly conclude approximation properties of the localized projection $\ProjVhmsktilde$ based on the previous lemmas.
\begin{corollary} \label{lem:IfProjVhmsktilde}
  Let $t \in [0,T]$ and $\varphi \in C^1([0,T],\Vhyph)$. There is a constant $0 < \mu < 1$ depending on $\frac{\alc}{\auc}$ such that
  \begin{align}
    \norm{\bigl(\Id - \ProjVhmsktildet\bigr) \varphi(t)}_{\Vhyp} \lesssim \parbox{50ex}{$H \norm{\Ophypsech(t) \varphi}_{\Hhyp} + \patchsize^{\spacedim/2} \mu^\patchsize \norm{\varphi(t)}_{\Vhyp},$} \label{eq:IdProjmsktildeVh} \\
    \norm{\bigl(\Id - \ProjVhmsktildet\bigr) \varphi(t)}_{\Hhyp} \lesssim \parbox{50ex}{$\Bigl(H + \patchsize^{\spacedim/2} \mu^\patchsize\Bigr) \norm{\varphi(t)}_{\Vhyp},$} \label{eq:IdProjmsktildeHh} \\
    \begin{aligned}
      \norm{\pt \bigl((\Id - \ProjVhmsktilde) \varphi\bigr)(t)}_{\Vhyp} &\lesssim \parbox{50ex}{$H \Bigl(\norm{\Ophypsech(t) \varphi(t)}_{\Hhyp} + \norm{\pt \bigl( \Ophypsech(t) \varphi(t) \bigr)}_{\Hhyp}\Bigr)$} \\
      &\hphantom{\leq{}}\parbox{50ex}{$+ H^{-\spacedim/2} \mu^\patchsize \bigl(\norm{\varphi(t)}_{\Vhyp} + \norm{\pt \varphi(t)}_{\Vhyp}\bigr),$}
    \end{aligned} \label{eq:ptIdProjmsktildeVh} \\
    \norm{\pt \bigl((\Id - \ProjVhmsktilde) \varphi\bigr)(t)}_{\Hhyp} \lesssim \parbox{50ex}{$\Bigl( H + H^{-\spacedim/2} \mu^\patchsize\Bigr) \Bigl(\norm{\varphi(t)}_{\Vhyp} + \norm{\pt \varphi(t)}_{\Vhyp}\Bigr).$} \label{eq:ptIdProjmsktildeHh} \\[-4ex]
  \end{align}
  The hidden constant might depend on $\alc$, $\auc$, $\ptauc$, and the mesh regularity, but not on the variations of $\aeps$.
\end{corollary}
We emphasize that \cref{eq:IdProjmsktildeVh} and \cref{eq:IdProjmsktildeHh} are also valid for the standard localized multiscale projection $\ProjVhmsk$, cf. \cite[Lem.~3.5]{MalP18}. However, up to our knowledge it is unclear whether this also extends to the estimates \cref{eq:ptIdProjmsktildeVh} and \cref{eq:ptIdProjmsktildeHh} for the time derivatives, even with \Cref{lem:ProjfProjfk} at hand. To circumvent this, in our analysis we rely on the non-orthogonal localized multiscale projection $\ProjVhmsktilde$ instead.

\subsection{Analysis of fully discrete localized orthogonal decomposition}
Finally, we present the proof of our main result.
\begin{proof}[Proof of \Cref{thm:backwardEulerError}]
  Using a similar notation as in \cref{eq:notationDiscreteOps}, i.e.,
  \begin{align}
    \solhyphd[m] &= \solhyph(\timed{m}), & \Ophyphd[m] &= \Ophyph(\timed{m}), & \rhshyphd[m] &= \rhshyph(\timed{m}), & m &= 0,\dots,N,
  \end{align}
  the solution $\solhyph = (\solhypsech, \pt \solhypsech)$ of \cref{eq:Cauchyhypdisc} satisfies the perturbed scheme
  \begin{align}
    \solhyphd[n+1] &= \solhyphd[n] + \tau \Ophyphd[n+1] \solhyphd[n+1] + \tau \rhshyphd[n+1] + \tau \defectEuler[n+1], & n &=0,\dots,N-1,
  \end{align}
  with a defect given by
  \begin{align}
    \defectEuler[n+1] &= \tfrac{1}{\tau} (\solhyphd[n+1] - \solhyphd[n]) + \pt \solhyph(\timed{n+1}).
  \end{align}
  Thus, the multiscale error $\errmskd[n] = \solhypmskd[n] - \RefXhmsktilded[n] \solhyphd[n]$ satisfies the error recursion
  \begin{align}
    \errmskd[n+1] - \ProjXhmskd[n+1] \errmskd[n] &= (\ProjXhmskd[n+1] - \Id) \RefXhmsktilded[n] \solhyphd[n] + \tau \Ophypmskd[n+1] \errmskd[n+1] + \tau \errmskrhsd[n+1],
  \end{align}
  with the right-hand side
  \begin{align} \label{eq:rhsErrorFull}
    \errmskrhsd[n+1] &= \bigl( \Ophypmskd[n+1] \RefXhmsktilded[n+1] - \Ophyphd[n+1] \bigr) \solhyphd[n+1] + \rhshypmskd[n+1] - \rhshyphd[n+1] - \defectEuler[n+1] \\
    &\hphantom{={}}+ \tfrac{1}{\tau} \int_{\timed{n}}^{\timed{n+1}} \pt \Bigl( (\Id - \RefXhmsktildet[t]) \solhyph(t) \Bigr) \dt.
  \end{align}
  Taking the inner product in $\Xhypmskd[n+1]$ with $\errmskd[n+1]$, we obtain with \cref{eq:OphyphSkew}, \cref{eq:notationDiscreteOps}, and Young's inequality
  \begin{align} \label{eq:errorDifferenceEuler}
    \norm{\errmskd[n+1]}_{\Xhypd[n+1]}^2 - \norm{\errmskd[n]}_{\Xhypd[n+1]}^2 &\leq 2 \tau \sprod{\Xhypd[n+1]}{\errmskrhsd[n+1]}{\errmskd[n+1]}.
  \end{align}

  To further bound the right-hand side, we first obtain due to the stability of $\ProjVhmsktilde$, \cref{eq:ProjfProjfk}, and \cref{eq:IdProjmsktildeVh} for $\varphi_h \in \Vhyph$ and $\psi_H \in \VhypH$ the estimate
  \begin{align} \label{eq:EstimatePhihPsiH}
    \begin{aligned}
      \MoveEqLeft \sprod{\Vhypd[n+1]}{\bigl( \ProjVhmsktilded[n+1] - \Id \bigr) \varphi_h}{(\ProjVhfd[n+1] - \ProjVhfkd[n+1]) \psi_H} \\
      &\lesssim \patchsize^{\spacedim/2} \mu^\patchsize \min\left\{\norm{\varphi_h}_{\Vhyp}, \Bigl( H \norm{\Ophypsech(\timed{n+1}) \varphi_h}_{\Hhyp} + \patchsize^{\spacedim/2} \mu^\patchsize \norm{\varphi_h}_{\Vhyp} \Bigr) \right\} \norm{\psi_H}_{\Vhyp}.
    \end{aligned}
  \end{align}
  We denote the finite element parts of the multiscale error $\errmskd[n+1] = (\errmsseckd[n+1], \pterrmsseckd[n+1])$ by $e_{H,\solhypsec}^{n+1}, e_{H,\ptsolhypsec}^{n+1} \in \VhypH$, i.e., we have
  \begin{align}
    \errmsseckd[n+1] &= (1 - \ProjVhfkd[n+1])e_{H,\solhypsec}^{n+1}, & \pterrmsseckd[n+1] &= (1 - \ProjVhfkd[n+1])e_{H,\ptsolhypsec}^{n+1}.
  \end{align}
  Based on an inverse estimate, cf. \cite[Thm.~4.5.11]{BreS08}, and the stability \cref{eq:InterpolhHBounds} and \cref{eq:InterpolhHStableHhyp} of $\InterpolH$ in $\Vhyp$ and $\Hhyp$, respectively, we further have
  \begin{align} \label{eq:boundsCoarseErrorParts}
    \norm{e_{H,\solhypsec}^{n+1}}_{\Vhyp} &\lesssim \norm{\errmsseckd[n+1]}_{\Vhyp}, & \norm{e_{H,\ptsolhypsec}^{n+1}}_{\Vhyp} &\lesssim H^{-1} \norm{\pterrmsseckd[n+1]}_{\Hhyp}.
  \end{align}
  Thus, for the choice $\patchsize \geq \bigl(1 + \frac{\spacedim}{2}\bigr)\abs{\log H}/\abs{\log \mu}$ we conclude with the Cauchy--Schwarz inequality, \cref{eq:IdProjVhmsktildetIdentity}, and \cref{eq:EstimatePhihPsiH} for the first term of \cref{eq:rhsErrorFull} 
  \begin{align}
    \MoveEqLeft \sprod{\Xhypd[n+1]}{\bigl( \Ophypmskd[n+1] \RefXhmsktilded[n+1] - \Ophyphd[n+1] \bigr) \solhyphd[n+1]}{\errmskd[n+1]} \\
    &= \sprod{\Vhypd[n+1]}{\bigl( \ProjVhmsktilded[n+1] - \Id\bigr) \ptsolhypsechd[n+1]}{(\ProjVhfd[n+1] - \ProjVhfkd[n+1])e_{H,\solhypsec}^{n+1}} \\
    &\hphantom{={}}- \sprod{\Vhypd[n+1]}{\bigl( \ProjVhmsktilded[n+1] - \Id\bigr) \solhypsechd[n+1]}{(\ProjVhfd[n+1] - \ProjVhfkd[n+1])e_{H,\ptsolhypsec}^{n+1}} \\
    &\lesssim H \Bigl( \norm{\solhypsechd[n+1]}_{\Vhyp} + \norm{\pt \solhypsechd[n+1]}_{\Vhyp} + \norm{\Ophypsech(\timed{n+1}) \solhypsechd[n+1]}_{\Hhyp} \Bigr) \norm{\errmskd[n+1]}_{\Xhypd[n+1]}.
  \end{align}
  Moreover, \cref{eq:notationDiscreteOps} implies
  \begin{align} \label{eq:rhshypmshvanishes}
    \sprod{\Xhypd[n+1]}{\rhshypmskd[n+1] - \rhshyphd[n+1]}{\errmskd[n+1]} &= 0.
  \end{align}
  For the defect, we deduce with Taylor's theorem
  \begin{align}
    \norm{\defectEuler[n+1]}_{\Xhypd[n+1]} &\lesssim \tau \bigl(\norm{\ptt \solhypsech}_{L^\infty([0,T],\Vhyp)} + \norm{\pttt \solhypsech}_{L^\infty([0,T],\Hhyp)}\bigr).
  \end{align}
  Finally, we apply \cref{eq:ptIdProjmsktildeVh} and \cref{eq:ptIdProjmsktildeHh} together with the norm equivalence \cref{eq:VVtnormequivalence} to show
  \begin{align}
    \norm{\pt \Bigl( \bigl( \Id - \RefXhmsktildet \bigr) \solhyph(t) \Bigr)}_{\Xhypt} &\lesssim H \Bigl(\norm{\Ophypsech(t) \solhypsech(t)}_{\Hhyp} + \norm{\pt \bigl( \Ophypsech(t) \solhypsech(t) \bigr)}_{\Hhyp} \\
    &\hphantom{\leq C H \Bigl(}+ \norm{\solhypsech(t)}_{\Vhyp} + \norm{\pt \solhypsech(t)}_{\Vhyp} + \norm{\ptt \solhypsech(t)}_{\Vhyp}\Bigr).
  \end{align}
  Collecting all results in \cref{eq:rhsErrorFull}, we obtain with \cref{eq:boundOphypsechsolhypsech}, \cref{eq:boundptOphypsechsolhypsech}, and Young's inequality
  \begin{align}
    \norm{\errmskd[n+1]}_{\Xhypd[n+1]}^2 - \norm{\errmskd[n]}_{\Xhypd[n+1]}^2 &\leq \tau \norm{\errmskd[n+1]}_{\Xhypd[n+1]}^2 + \tau \widehat{\mathcal{C}}_{\solhypsec,\rhshypsec}^{\,2} \bigl( \tau +  H \bigr)^2,
  \end{align}
  for some $\eps$-independent constant $\widehat{\mathcal{C}}_{\solhypsec,\rhshypsec} > 0$ depending on the regularity assumptions \cref{eq:regularityforEulererrornonautonomouswave}. Further, since \cref{eq:aepsBounds} implies
  \begin{align} \label{eq:normexponential}
    \begin{aligned}
      \norm{\errmskd[n]}_{\Xhypd[n+1]}^2 &= \norm{\errmskd[n]}_{\Xhypd[n]}^2 + \int_{\timed{n}}^{\timed{n+1}} \sprod{\Hhyp}{\pt \aeps(t) \grad \errmsseckd[n]}{\grad \errmsseckd[n]} \dt \leq \expp{\ptauc \tau} \norm{\errmskd[n]}_{\Xhypd[n]}^2,
    \end{aligned}
  \end{align}
  we have
  \begin{align}
    \norm{\errmskd[n+1]}_{\Xhypmskd[n+1]}^2 - \expp{\ptauc \tau} \norm{\errmskd[n]}_{\Xhypmskd[n]}^2 &\leq \tau \norm{\errmskd[n+1]}_{\Xhypmskd[n+1]}^2 + \tau \widehat{\mathcal{C}}_{\solhypsec,\rhshypsec}^{\,2} \bigl( \tau +  H \bigr)^2.
  \end{align}
  Multiplying with $\expp{-\ptauc \timed{n+1}}$ and using this identity recursively for $n,\dots,0$, we get
  \begin{align}
    \expp{-\ptauc \timed{n+1}} \norm{\errmskd[n+1]}_{\Xhypd[n+1]}^2 &\leq \norm{\errmskd[0]}_{\Xhypd[0]}^2 + \tau \sum_{r=0}^n \expp{-\ptauc \timed{r+1}} \norm{\errmskd[r+1]}_{\Xhypd[r+1]}^2 + \timed{n+1} \widehat{\mathcal{C}}_{\solhypsec,\rhshypsec}^{\,2} \bigl( \tau +  H \bigr)^2.
  \end{align}
  Finally, a discrete version of Gronwall's inequality shows
  \begin{align}
    \norm{\errmskd[n+1]}_{\Xhypd[n+1]}^2 &\leq \expp{\mathcal{C}_\acoeff \timed{n+1}} \Bigl( \norm{\errmskd[0]}_{\Xhypd[0]}^2 + \timed{n+1} \widehat{\mathcal{C}}_{\solhypsec,\rhshypsec}^{\,2} \bigl( \tau +  H \bigr)^2 \Bigr).
  \end{align}
  This concludes the proof, since with the notation $\solhyphinit = (\solhypsechinit, \ptsolhypsechinit)$ the initial error satisfies due to \cref{eq:notationDiscreteOps} and \cref{eq:IdProjmsktildeHh} the estimate
  \begin{align}
    \norm{\errmskd[0]}_{\Xhypd[0]}^2 &= \norm{\bigl(\ProjHhmskt[0] - \ProjVhmsktildet[0]\bigr) \ptsolhypsechinit}_{\Hhyp}^2 \\
    &\leq \sprod{\Hhyp}{\bigl(\Id - \ProjVhmsktildet[0]\bigr) \ptsolhypsechinit}{\bigl(\ProjHhmskt[0] - \ProjVhmsktildet[0]\bigr) \ptsolhypsechinit} \\
    &\lesssim H \norm{\ptsolhypsechinit}_{\Vhyp} \norm{\errmskd[0]}_{\Xhypd[0]}.\qedhere
  \end{align}
\end{proof}

\section{Practical aspects} \label{sec:practicalAspects}
In the previous sections, we proved that our fully discrete scheme \cref{eq:backwardEulerLOD} for the wave equation with time-dependent multiscale coefficients is wellposed and first-order convergent. However, since the multiscale coefficient $\aeps$ is time-dependent, we have to recompute all correctors \cref{eq:defProjVhfKt} in each time step, which might become rather time-consuming in practice. To circumvent this, we modify \cref{eq:backwardEulerLOD} in the following two subsections. First, we propose a Petrov--Galerkin variant in the spirit of \cite{ElfGH15} to avoid products of multiscale functions thereby reducing the communication in the assembly of the mass and stiffness matrices of the discrete system. Furthermore, this allows for an adaptive update strategy, which significantly improves the computational efficiency, cf.~\cite{HelM19,HelKM20}.
We work with the backward Euler time stepping scheme as before, but emphasize that the following derivations can be easily transferred to the implicit midpoint rule mentioned in~\Cref{rem:implicitMidpoint}.

\subsection{Petrov--Galerkin formulation}
We introduce the Petrov--Galerkin formulation in the matrix-vector-notation. To do so, we denote the nodal basis of the coarse finite element space $\FESpaceH$ by $\{\lambda_j\}_{j=1}^J$, i.e., there are grid points $\{x_j\}_{j=1}^J \subset \domain$ such that $\lambda_i(x_j) = \delta_{ij}$, $i, j=1,\dots,J$, where $\delta_{ij}$ denotes the Kronecker delta. Due to the structure of the localized multiscale space $\Vhypmskt$ \cref{eq:defFESpacemskt}, we identify any time-dependent function $\zeta \colon [0,T] \to \Vhypmskt$ with its time-dependent vector-valued representation $\underline{\zeta} = (\underline{\zeta}_1, \dots, \underline{\zeta}_J) \colon [0,T] \to \R^J$, which is uniquely characterized by
\begin{align}
  \zeta(t) &= \sum_{j=1}^J \bigl(\Id-\ProjVhfkt\bigr) \lambda_j \underline{\zeta}_j(t), & t &\in [0,T].
\end{align}

For the Petrov--Galerkin variant of \cref{eq:backwardEulerLOD} we directly use the coarse finite element space $\VhypH$ instead of the time-dependent localized multiscale space $\Vhypmskt$ as the test space. To this end, we introduce for $m=1,\dots,N$ the mass and stiffness matrices
\begin{align}
  \begin{aligned}
    \underline{M}_{ij}^m &= \sum_{K \in \TriaH} (\underline{M}_K^m)_{ij} = \sum_{K \in \TriaH} \sprod{L^2(\patch{K})}{\lambda_i}{(\Id-\ProjVhfkKd[m])\lambda_j}, \\
    \underline{A}_{ij}^m &= \sum_{K \in \TriaH} (\underline{A}_K^m)_{ij} = \sum_{K \in \TriaH} \sprod{L^2(\patch{K})}{\aepsd{m}\nabla\lambda_i}{\nabla \bigl((\Id-\ProjVhfkKd[m])\lambda_j\bigr)\bigr)},
  \end{aligned}
  && i,j &=1,\dots,J,
\end{align}
based on the short notation $\aepsd{m} = \aeps(\timed{m},\cdot)$ and $\ProjVhfkKd[m] = \ProjVhfkKt[\timed{m}]$ for the multiscale coefficient and the localized correctors \cref{eq:defProjVhfkKt}, respectively. For the first-order system, we further define
\begin{align} 
 \underline{\mathcal{A}}^m &= \begin{pmatrix} 0 & \underline{M}^m \\ -\underline{A}^m & 0 \end{pmatrix}, &
 \underline{\mathcal{M}}^m &= \begin{pmatrix} \underline{M}^m & 0 \\ 0 & \underline{M}^m \end{pmatrix}.
\end{align}

Then, the solution of the Petrov--Galerkin variant of \cref{eq:backwardEulerLOD} is given by
\begin{align}
  z^m &= \sum_{j=1}^J \bigl(\Id-\ProjVhfkd[m]\bigr) \lambda_j \underline{z}_j^m, & m&=0,\dots,N,
\end{align}
where the vector-valued representations satisfy the recursion
\begin{align} \label{eq:backwardEulerPGLOD}
  \underline{\mathcal{M}}^{n+1}\underline{z}_i^{n+1} &= \underline{\mathcal{M}}^{n+1}\underline{z}_i^n + \tau \underline{\mathcal{A}}_{ij}^{n+1} \underline{z}_j^{n+1} + \tau \underline{\mathcal{F}}_{i}^{n+1}, & n &=0,\dots,N-1.
\end{align}
For the initial value, we set $\underline{z}^0 = \InterpolH \solhyphinit$. Further, we set $\underline{\mathcal{F}}_{i}^{n+1}=\sprod{L^2(\domain)}{f_h(\timed{n+1}, \cdot)}{\lambda_i}$. 

\begin{remark}
	Note that in~\cref{eq:backwardEulerPGLOD}, we slightly modified the projection of $\underline{z}_i^n$ in comparison to~\cref{eq:backwardEulerLOD}.
	Namely, we use the finescale projection $\ProjVhfkKd[n+1]$ instead of projecting $(\Id-\ProjVhfkKd[n])\lambda_i\underline{z}_i^n$ onto $\Vhypmskd[n+1]$.
	This, however, is justified by the error estimates obtained through~\cref{eq:ProjfProjfk}.
	In fact, one can even use the standard finite element mass matrix on the coarse mesh, see~\cite{MaiP19}, and mass lumping is possible as well, see~\cite{GeeM21*}. 
\end{remark}

\subsection{Adaptive update strategy}\label{subsec:adapupdate}
As a remedy for the recomputation of all correctors in each time step, we propose an adaptive update strategy for the correctors and LOD matrices that will be investigated in detail in the numerical experiments.
The strategy is based on the following local error indicators. Let $\aepsd{i}=\aeps(\timed{i}, \cdot)$ and $\aepsd{n}=\aeps(\timed{n}, \cdot)$ be the coefficients at two different time instances with associated correctors $\ProjVhfkd[i]$ and $\ProjVhfkd[n]$. 
For any element $K$, we denote by
\[\apatch[k]{K}(t)=\abs{\patch[k]{K}}^{-1}\int_{\patch[k]{K}} \aeps(t, x) \dx\]
the local average of $\aeps$ and by 
\[\hat{\hat{a}}_{\eps,\patch[k]{K}}(t, x)=\aeps(t,x)|_{\patch[k]{K}}/\apatch[k]{K}(t)\]
the locally scaled version of $\aeps$.
Similar to \cite{HelM19,HelKM20}, we define for each element $K \in \TriaH$ the local error indicator
\begin{equation}
	\begin{aligned}
	E_K(\aepsd{i}, \aepsd{n})^2=\norm{\hat{\hat{a}}_\eps^i (\hat{\hat{a}}_\eps^n)^{-1}}_{L^\infty(K)}\sum_{K^\prime\in \patch[k]{K}}&\norm{(\hat{\hat{a}}_\eps^i)^{-1/2}(\hat{\hat{a}}_\eps^i-\hat{\hat{a}}_\eps^n)(\hat{\hat{a}}_\eps^n)^{-1/2}}^2_{L^\infty(K^\prime)}\\
	&\cdot\max_{v|_K, v\in V_H}\frac{\norm{(\hat{\hat{a}}_\eps^i)^{1/2}(\chi_K\nabla v-\nabla \ProjVhfkK^i v)}^2_{L^2(K^\prime)}}{\norm{(\hat{\hat{a}}_\eps^i)^{1/2}\nabla v}^2_{L^2(K)}},
	\end{aligned}
\end{equation}
where $\chi_K$ denotes the characteristic function of $K$.
We emphasize that we consider locally scaled versions of the coefficients in contrast to the original indicators \cite{HelM19}.
This scaling is important in the present setting of time-dependent coefficients: As discussed in \Cref{eq:timeDependentMultiscaleSpaces}, the correctors and multiscale spaces do not need to change over time in the case of coefficients with ``tensor-product structure'', i.e., $\aeps(t,x)=a_{1,\eps}(x)a_2(t)$.
The local error indicator defined above reflects this fact in the sense that $E_K(\aepsd{i}, \aepsd{n})=0$ for all $K$, $i$ and $n$ for such tensor-product coefficients $\aeps$.

\begin{remark}
	We choose to scale the coefficient by dividing with its local average. Other scalings, such as the local minimum or maximum, are equally possible.
	To extend the procedure to matrix-valued coefficients, one can, for instance, scale with the local average of the trace of $\aeps$.
\end{remark}

We now explain the adaptive algorithm which we illustrate for the LOD stiffness matrix $\underline{A}^{n}$ in the $n$th time step.
For the first time step, all correctors $\widehat{R}_{\textup{f},k}= \ProjVhfkd[1]$ are computed based on $\widehat{a}^{1}=\aeps(\timed{1}, \cdot)$, as described in the previous section.
In the subsequent time steps, we evaluate $E_K(\widehat{a}^{n-1}, \aepsd{n})$ for all elements $K$.
Given a prescribed tolerance \texttt{tol}, we mark all elements with $E_K(\widehat{a}^{n-1}, \aepsd{n})\geq \texttt{tol}$ and compute a new corrector $\ProjVhfkKd[n]$ for these elements based upon $\aepsd{n}$.
Otherwise, the available corrector from the previous time step(s) is used.
We then set
\begin{equation}
\begin{aligned}
	\widehat{R}_{\textup{f}, K, k}^n=\begin{cases}
		\ProjVhfkKd[n], \qquad K\text{ marked},\\
		\widehat{R}_{\f,k,K}^{n-1},\qquad \text{else},
	\end{cases}
\qquad \text{and}\qquad
\widehat{a}_{K}^n=\begin{cases}
	\aepsd{n}|_{\patch[k]{K}}, \hphantom{\widehat{a}_{K}^{n-1}} K\text{ marked},\\
	\widehat{a}_{K}^{n-1}, \hphantom{\aepsd{n}|_{\patch[k]{K}}} \text{else},
\end{cases}
\end{aligned}
\end{equation}
and assemble the stiffness matrix contributions as
\[(\underline{\widehat{A}}_K)_{ij}^n=\sprod{\patch[k]{K}}{\widehat{a}^n_K\nabla \lambda_i}{ \nabla(\Id-\widehat{R}_{\f,k,K}^n)\lambda_j}.\]
This simply means that the LOD stiffness matrix is defined as a mixture of newly computed correctors as well as local coefficients of the current time step and reused correctors and local coefficients of previous time step(s).
A similar procedure can be applied to calculate the other LOD matrices if necessary.
Finally, we note that in the extreme cases $\texttt{tol}=0$ or $\texttt{tol}=\infty$, we obtain the Petrov--Galerkin variant of the method from \Cref{subsec:fullyDiscreteScheme} or a time-stepping with fixed multiscale space based upon $a(\timed{1}, \cdot)$, respectively.

\begin{remark}[Practical choice of \texttt{tol}]\label{rem:tolerance}
	Since the absolute value of the error indicator is hard to predict in practice, we suggest the following choice of the tolerance that is also used in our numerical experiments: Fix a tolerance factor $\zeta_{\mathrm{tol}}\in [0,1]$ and set the tolerance  in the $n$th time step to \[\mathtt{tol}= \bigl(\min_{K\in \mathcal T_H} E_K^n\bigr)+\zeta_{\mathrm{tol}}\, \bigl(\max_{K\in \mathcal T_H} E_K^n-\min_{K\in \mathcal T_H}E_K^n\bigr),\]
	where $E_K^n= E_K(\widehat{a}^{n-1}, \aepsd{n})$ denotes the error indicator in the $n$th step.
\end{remark}

It was shown in \cite{HelM19} that $\ProjVhfkK^n-\widehat{R}_{\f,k,K}^n$ can be bounded by $E_K(\widehat{a}^n, \aepsd{n})$ and thereby by \texttt{tol}.
Similar to the estimate for the consistency error in \cite[Thm.~4.1]{HelKM20}, one can show that
\[\norm{\underline{A}^n-\underline{\widehat{A}}^n}\lesssim k^{d/2}\,  \mathtt{tol}\]
holds for all $n$ in a suitable matrix norm.

Let $\widehat{y}_{k,\mathrm{ms}}^n\in (1-\widehat{R}_{\f,k}^n)\XhypH$ be the LOD solution computed with the adaptive update strategy.
From \Cref{thm:backwardEulerError} and a perturbation argument, we expect the following error estimate
\[\norm{\widehat{y}_{k,\mathrm{ms}}^n-y_h^n}_{\mathcal X}\leq \mathcal{C}_{\solhypsec,\rhshypsec}\, \expp{\mathcal{C}_\acoeff \timed{n}}\, (\tau + H + \mathtt{tol}).\]
In fact, our numerical experiments show that the LOD with adaptive update strategy still converges with (spatial) rate $H$ if $\mathtt{tol}$ is chosen small enough and consequently the consistency error is sufficiently small in comparison to the discretization error.

\section{Numerical examples} \label{sec:NumExp}
In this section, we illustrate the theoretical error estimate and the adaptive update strategy by numerical examples.
We implemented the Petrov--Galerkin method with the implicit midpoint rule as the time integration scheme using the python module \texttt{gridlod} \cite{HelK19}. 
The code to reproduce the examples below is publicly available at \url{https://github.com/BarbaraV/gridlod-timedependent}.

Throughout the experiments, we consider the wave equation with time-dependent multiscale coefficients on the spatial domain $\domain=[0,1]^2$ and for the time interval $[0,T]$ with final time $T=1$.
We always use homogeneous initial conditions, i.e., $\solhypsecinit=0$ and $\ptsolhypsecinit=0$.
The numerical experiments study the relative error between a finescale finite element (reference) solution and the LOD solution measured in the energy norm of $\Xhyp$ at final time $T=1$.

\subsection{Spatial and temporal convergence}\label{subsec:exp1}
In our first numerical experiment, we illustrate spatial and temporal convergence rates of the LOD where all correctors are updated in every time step. 
We choose a periodic time-dependent multiscale coefficient as
\[\aeps(t,x)=(3+\sin\bigl(2\pi\frac{x_1}{\varepsilon}\bigr)+\sin(2\pi t))(3+\sin\bigl(2\pi\frac{x_2}{\varepsilon}\bigr)+\sin(2\pi t))\]
with $\varepsilon=2^{-7}$
and compare two right-hand sides with different spatial regularity, namely
\[f_1(t,x)=\begin{cases}
	20t+230t^2,\qquad\;\;\;\, x_1>0.4,\\
	100t+2300t^2,\qquad x_1<0.4,
\end{cases}\]
and \[f_2(t,x)=20t(x_1-x_1^2)(x_2-x_2^2)+230t^2(x_1-x_1^2+x_2-x_2^2).\]
Note that $f_1\in C^\infty(0,T; L^2(\domain))$ and $f_2\in C^\infty(0,T; H^1(\domain))$.
The reference solution is computed using the finite element method on a fine mesh with mesh width $h=2^{-9}$ and the implicit midpoint rule with step size $\tau =2^{-7}$.

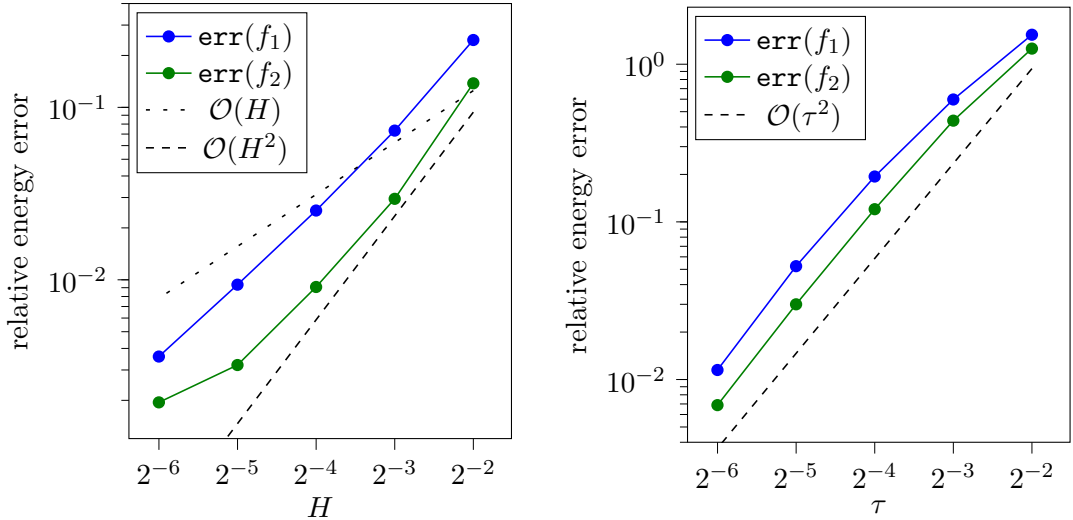
\begin{figure}
	\centering
	\begin{minipage}{0.5\textwidth}
		\begin{tikzpicture}
			
			\begin{axis}[
				height=1.0\textwidth,
				width=0.9\textwidth,
				log basis x={2},
				log basis y={10},
				tick align=outside,
				tick pos=left,
				x grid style={white!69.0196078431373!black},
				xlabel={$H$},
				xlabel style={at={(0.5,-0.02)}},
				xmin=0.012, xmax=0.35,
				xmode=log,
				xtick style={color=black},
				y grid style={white!69.0196078431373!black},
				ylabel={relative energy error},
				ylabel style={at={(-0.04,0.5)}},
				ymin=0.0012, ymax=0.4,
				ymode=log,
				ytick style={color=black},
				legend style={at={(0.02,0.98)},anchor=north west}
				]
				\addplot [semithick, blue, mark=*, mark size=2, mark options={solid}]
				table {%
					0.25 0.245773933509689
					0.125 0.0733268175824731
					0.0625 0.0252186456156111
					0.03125 0.00936959605581148
					0.015625 0.00358906800823274
				};
				\addlegendentry{$\mathtt{err} (f_1)$}
				\addplot [semithick, green!50.1960784313725!black, mark=*, mark size=2, mark options={solid}]
				table {%
					0.25 0.137869540378424
					0.125 0.0294840150376002
					0.0625 0.00908510421570741
					0.03125 0.00320452604705976
					0.015625 0.00194646973993396
				};
				\addlegendentry{$\mathtt{err} (f_2)$}
				\addplot [semithick, black, dash pattern=on 1.5pt off 5pt]
				table {%
					0.25 0.125
					0.125 0.0625
					0.0625 0.03125
					0.03125 0.015625
					0.015625 0.0078125
				};
				\addlegendentry{$\mathcal{O}(H)$}
				\addplot [semithick, black, dashed]
				table {%
					0.25 0.09375
					0.125 0.0234375
					0.0625 0.005859375
					0.03125 0.00146484375
					0.015625 0.0003662109375
				};
				\addlegendentry{$\mathcal{O}(H^2)$}
			\end{axis}
			
		\end{tikzpicture}
	\end{minipage}%
	\begin{minipage}{0.5\textwidth}
	  \begin{tikzpicture}
	  	
	  	\begin{axis}[
	  		height=1.0\textwidth,
	  		width=0.9\textwidth,
	  		log basis x={2},
	  		log basis y={10},
	  		tick align=outside,
	  		tick pos=left,
	  		x grid style={white!69.0196078431373!black},
	  		xlabel={$\tau$},
	  		xlabel style={at={(0.5,-0.02)}},
	  		xmin=0.012, xmax=0.35,
	  		xmode=log,
	  		xtick style={color=black},
	  		y grid style={white!69.0196078431373!black},
	  		ylabel={relative energy error},
	  		ylabel style={at={(-0.04,0.5)}},
	  		ymin=0.004, ymax=2.3,
	  		ymode=log,
	  		ytick style={color=black},
	  		legend style={at={(0.02,0.98)},anchor=north west}
	  		]
	  		\addplot [semithick, blue, mark=*, mark size=2, mark options={solid}]
	  		table {%
	  			0.25 1.5382556846222
	  			0.125 0.597484258178004
	  			0.0625 0.193991308839704
	  			0.03125 0.0523529458680587
	  			0.015625 0.0114938436587302
	  		};
	  		\addlegendentry{$\mathtt{err} (f_1)$}
	  		\addplot [semithick, green!50.1960784313725!black, mark=*, mark size=2, mark options={solid}]
	  		table {%
	  			0.25 1.25620624152684
	  			0.125 0.438527390878514
	  			0.0625 0.12034311100664
	  			0.03125 0.0299959456843265
	  			0.015625 0.00688015517075078
	  		};
	  		\addlegendentry{$\mathtt{err} (f_2)$}
	  		\addplot [semithick, black, dashed]
	  		table {%
	  			0.25 0.9375
	  			0.125 0.234375
	  			0.0625 0.05859375
	  			0.03125 0.0146484375
	  			0.015625 0.003662109375
	  		};
	  		\addlegendentry{$\mathcal{O}(\tau^2)$}
	  	\end{axis}
	  	
	  \end{tikzpicture}
	\end{minipage}%
	\caption{Convergence of the relative energy error for the experiments of \Cref{subsec:exp1}, left: spatial convergence, right: temporal convergence.}
	\label{fig:exp1}
\end{figure}

We first fix $\tau=2^{-7}$ for the LOD and study the spatial convergence on meshes with $H=2^{-2}, 2^{-3}, \ldots, 2^{-6}$ and $k=1,2,2,3,3$.
\Cref{fig:exp1} (left) shows that for $f_1$, we obtain a bit more than linear convergence (rate of about $H^{3/2}$) and for $f_2$, the convergence is even of quadratic order.
Note that $f_1$ fulfills the regularity and compatibility requirements for our theoretical error estimates and the results underline the spatial convergence rates predicted.
Slightly more than linear convergence was also observed in numerical experiments for the autonomous wave equation, cf.~\cite{MaiP19}.
The better rate for $f_2$ is explained by its higher spatial regularity $H^1(\domain)$ as also shown theoretically in the time-invariant case \cite{AbdH17}.
These refined error estimates may carry over to the present non-autonomous case.
The slower convergence for the smallest $H$ for $f_2$ can probably be cured by a larger choice of $k$.

To study the temporal convergence, we fix $H=2^{-6}$ and $\patchsize=3$ and vary $\tau = 2^{-2}, 2^{-3}, \ldots, 2^{-6}$.
\Cref{fig:exp1} (right) shows that we obtain a quadratic rate in $\tau$ for both right-hand sides. 
This is the expected result for the implicit midpoint rule, cf.~\Cref{rem:implicitMidpoint}.
Altogether, this experiment clearly underlines the theoretically expected convergence rates for the LOD energy error when all correctors are updated in every step.

\subsection{Adaptive update strategy}\label{subsec:exp2}
We now study the adaptive update strategy presented in \Cref{subsec:adapupdate} and its influence on the energy error.
In the following we focus on the spatial convergence for fixed tolerance as well as the dependence of the error on the chosen tolerance.
We compare different multiscale coefficients, which are discontinuous in space.
We choose as right-hand side
\[f(t,x)=\sin(\pi x_1)\sin(\pi x_2)(5t+50t^2),\]
which has the same regularity properties as $f_2$ in the first example.
The reference finite element solution is again computed with $h=2^{-9}$ and $\tau = 2^{-7}$.

The first coefficient is
\[a_1(t,x)=(1+0.5\cos(9t))a_{\mathrm{disc}}(x),\quad \text{with}\quad a_{\mathrm{disc}}(x)=\begin{cases}
	10, \quad \frac{x}{\varepsilon}\in [0.25, 0.75]^2,\\
	1, \quad \quad \text{else},
\end{cases} \]
with $\varepsilon=2^{-7}$. The coefficient $a_\mathrm{disc}$ is visualized in \Cref{fig:exp2a} (left), where we chose $\varepsilon=2^{-5}$ for better visibility.
The structure of $a_1$ allows us to keep the correctors constant in time, cf.~\Cref{eq:timeDependentMultiscaleSpaces}, and only multiply the stiffness matrix with the global value $(1+0.5\cos(9\timed{n+1/2}))$ in the $n$th time step.
Hence, in this case the method is as efficient as for time-independent coefficients: We can pre-compute the multiscale basis (or the LOD stiffness matrix) and only have to solve a small linear system in each time step.
\Cref{fig:exp2a} (right) shows the spatial convergence for $H=2^{-2}, 2^{-3}, \ldots, 2^{-6}$, $k=1,2,2,3,3$, and $\tau =2^{-7}$, where we observe a quadratic rate as expected due to the spatial regularity of the right-hand side, as discussed in the previous section.
We emphasize that it is crucial to use the error indicator with scaled coefficients in this example since otherwise unnecessary updates of the correctors would be triggered.

\begin{figure}
	\centering
	\includegraphics[width=0.42\textwidth, trim=25mm 0mm 28mm 10mm, clip=true]{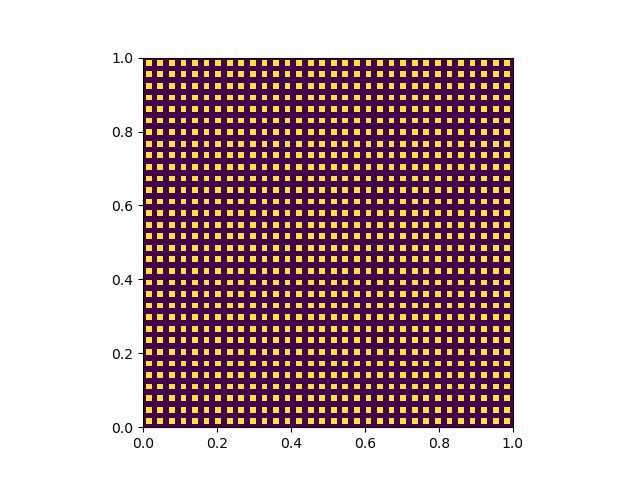}%
	\hspace*{0.02\textwidth}%
	\begin{tikzpicture}
		
		\begin{axis}[
			height=0.43\textwidth,
			width=0.53\textwidth,
			log basis x={2},
			log basis y={10},
			tick align=outside,
			tick pos=left,
			x grid style={white!69.0196078431373!black},
			xlabel={$H$},
			xlabel style={at={(0.5,-0.02)}},
			xmin=0.012, xmax=0.35,
			xmode=log,
			xtick style={color=black},
			y grid style={white!69.0196078431373!black},
			ylabel={relative energy error},
			ylabel style={at={(-0.04,0.5)}},
			ymin=0.00015, ymax=0.3,
			ymode=log,
			ytick style={color=black},
			legend style={at={(0.02,0.98)},anchor=north west}
			]
			\addplot [semithick, black, mark=*, mark size=2, mark options={solid}]
			table {%
				0.25 0.12919456976444
				0.125 0.0223567772679948
				0.0625 0.00584986822765361
				0.03125 0.00126165150192473
				0.015625 0.000286924794058873
			};
			\addlegendentry{$\mathtt{err} (a_1)$}
			\addplot [semithick, black, dashed]
			table {%
				0.25 0.05
				0.125 0.0125
				0.0625 0.003125
				0.03125 0.00078125
				0.015625 0.000195313
			};
			\addlegendentry{$\mathcal{O}(H^2)$}
		\end{axis}
		
	\end{tikzpicture}
	\caption{First experiment in \Cref{subsec:exp2}: coefficient $a_{\mathrm{disc}}$ for $\varepsilon=2^{-5}$ (left -- blue is $1$ and yellow is $10$) and spatial convergence of the LOD method for $a_1$ (right).}
	\label{fig:exp2a}
\end{figure}
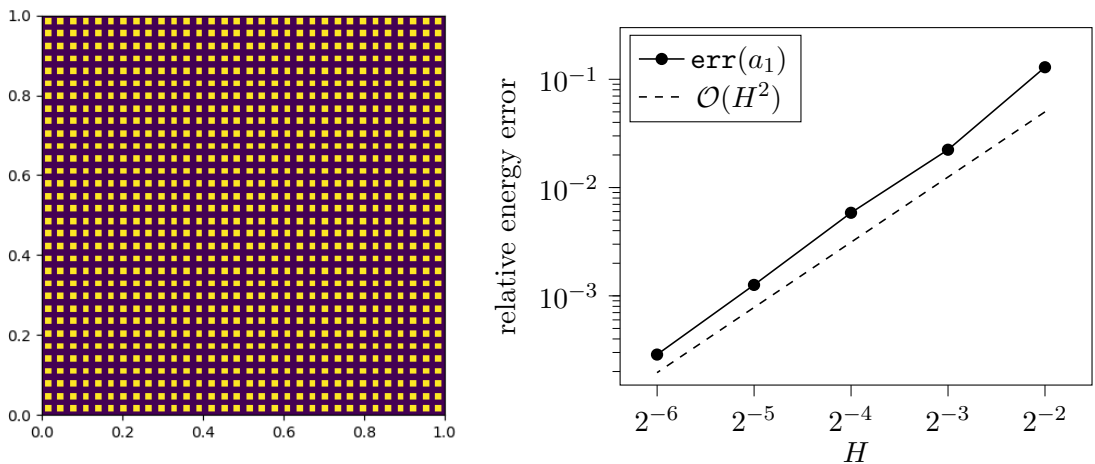

Now we consider two coefficients without this ``tensor-product'' structure in space and time to study the influence of the update strategy.
We consider
\[a_2(t,x)=\begin{cases}
	(1+0.5\cos(9t)) a_{\mathrm{disc}}(x), \qquad x\in [0.25, 0.75]^2,\\a_{\mathrm{disc}}(x), \qquad\qquad\qquad\qquad\quad \text{else},
\end{cases}\]
and
\[a_3(t,x)= a_{\mathrm{disc}}(x)+1+0.5\cos(9t)\]
with the same $a_{\mathrm{disc}}$ as for $a_1$.
We first update all correctors where the indicator has a value larger than the mean value of all error indicators, i.e., $\zeta_{\mathrm{tol}}=0.5$ as explained in \Cref{rem:tolerance}.
The spatial convergences for $H=2^{-2}, 2^{-3}, 2^{-4}, 2^{-5}$, $k=1,2,2,3$, and $\tau=2^{-6}$ are shown in \Cref{fig:exp2b} (left).
Except for the last mesh size, we still obtain at least linear convergence in $H$, while especially for $a_2$ we even have quadratic convergence.
The theoretically expected order is $\mathcal{O}(H^2)$ if we update all correctors in every time step and slower linear convergence as well as the stagnation for the last mesh size can be explained by the dominance of the updating error.
Still, the relative energy error is only a few percent for mesh sizes of $H=2^{-4}$ or $H=2^{-5}$, which is satisfactory in many applications. 
Moreover, we emphasize that by the adaptive strategy we only need to update $14.1\%$ or $50.6\%$ of the correctors on average in every step for $a_2$ or $a_3$, respectively.
As the main computation time is spent on the assembly of the LOD stiffness matrices, such an adaptive update strategy reduces the computational complexity considerably in comparison to the ``perfect'' case where correctors are computed in every time step. 
The smaller update percentage for $a_2$ is caused by the time modulation acting only in some part of the domain.
\Cref{fig:exp2b} (right) shows how the maximal tolerance (over all time steps) as well as the energy error evolve for different choices of the tolerance factor for fixed $H=2^{-5}$ and $k=3$.
We see that tolerance factors around $0.5$ seem to provide a good compromise between computational efficiency and accuracy.

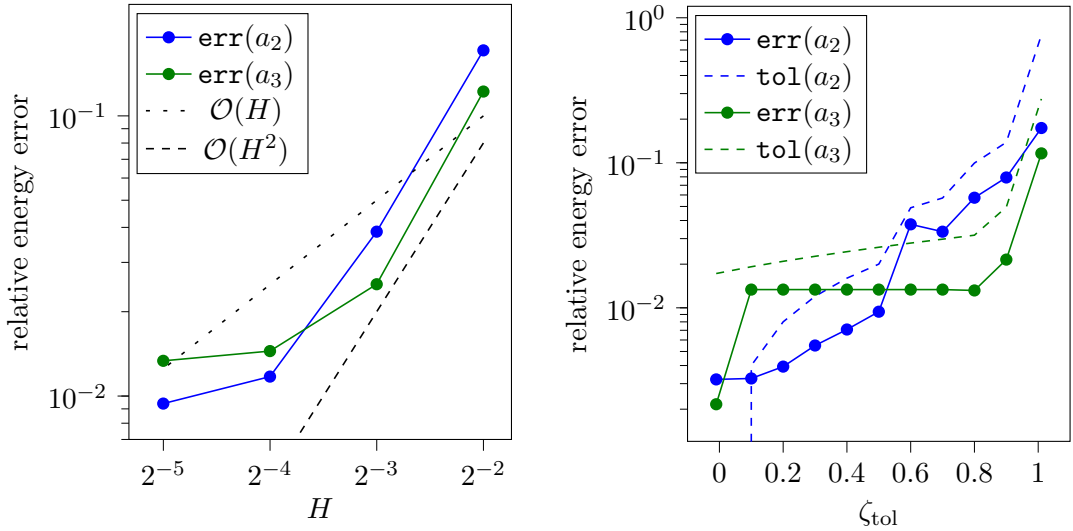
\begin{figure}
	\centering
	\begin{minipage}{0.5\textwidth}
	  \begin{tikzpicture}
	  	
	  	\begin{axis}[
	  		height=1.0\textwidth,
	  		width=0.9\textwidth,
	  		log basis x={2},
	  		log basis y={10},
	  		tick align=outside,
	  		tick pos=left,
	  		x grid style={white!69.0196078431373!black},
	  		xlabel={$H$},
	  		xlabel style={at={(0.5,-0.02)}},
	  		xmin=0.025, xmax=0.3,
	  		xmode=log,
	  		xtick style={color=black},
	  		y grid style={white!69.0196078431373!black},
	  		ylabel={relative energy error},
	  		ylabel style={at={(-0.04,0.5)}},
	  		ymin=0.007, ymax=0.25,
	  		ymode=log,
	  		ytick style={color=black},
	  		legend style={at={(0.02,0.98)},anchor=north west}
	  		]
	  		\addplot [semithick, blue, mark=*, mark size=2, mark options={solid}]
	  		table {%
	  			0.25 0.171353877509619
	  			0.125 0.038617799415264
	  			0.0625 0.0117342014940261
	  			0.03125 0.00940022493639302
	  		};
	  		\addlegendentry{$\mathtt{err} (a_2)$}
	  		\addplot [semithick, green!50.1960784313725!black, mark=*, mark size=2, mark options={solid}]
	  		table {%
	  			0.25 0.122146268890487
	  			0.125 0.0250680831959435
	  			0.0625 0.0144713959450606
	  			0.03125 0.0133576912757673
	  		};
	  		\addlegendentry{$\mathtt{err} (a_3)$}
	  		\addplot [semithick, black, dash pattern=on 1.5pt off 5pt]
	  		table {%
	  			0.25 0.1
	  			0.125 0.05
	  			0.0625 0.025
	  			0.03125 0.0125
	  		};
	  		\addlegendentry{$\mathcal{O}(H)$}
	  		\addplot [semithick, black, dashed]
	  		table {%
	  			0.25 0.08
	  			0.125 0.02
	  			0.0625 0.005
	  			0.03125 0.00125
	  		};
	  		\addlegendentry{$\mathcal{O}(H^2)$}
	  	\end{axis}
	  	
	  \end{tikzpicture}
	\end{minipage}%
	\begin{minipage}{0.5\textwidth}
	  \begin{tikzpicture}
	  	
	  	\begin{axis}[
	  		height=1.0\textwidth,
	  		width=0.9\textwidth,
	  		log basis y={10},
	  		tick align=outside,
	  		tick pos=left,
	  		x grid style={white!69.0196078431373!black},
	  		xlabel={$\zeta_{\mathrm{tol}}$},
	  		xlabel style={at={(0.5,-0.02)}},
	  		xmin=-0.1, xmax=1.1,
	  		xtick style={color=black},
	  		y grid style={white!69.0196078431373!black},
	  		ylabel={relative energy error},
	  		ylabel style={at={(-0.04,0.5)}},
	  		ymin=0.0012, ymax=1.2,
	  		ymode=log,
	  		ytick style={color=black},
	  		legend style={at={(0.02,0.98)},anchor=north west}
	  		]
	  		\addplot [semithick, blue, mark=*, mark size=2, mark options={solid}]
	  		table {%
	  			-0.01 0.00320887384970397
	  			0.1 0.00325828898539855
	  			0.2 0.00393212586938201
	  			0.3 0.00549705409482705
	  			0.4 0.00708712537162354
	  			0.5 0.00940022493639302
	  			0.6 0.0376367520078116
	  			0.7 0.0334954024471669
	  			0.8 0.0574396171892143
	  			0.9 0.0792190930534284
	  			1.01 0.173719719134373
	  		};
	  		\addlegendentry{$\mathtt{err} (a_2)$}
	  		\addplot [semithick, blue, dashed]
	  		table {%
	  			0.1 0.00000000000001
	  			0.1 0.00400277695467322
	  			0.2 0.00800555390934644
	  			0.3 0.0120083308640197
	  			0.4 0.0160111078186929
	  			0.5 0.0200138847733661
	  			0.6 0.0489551953331495
	  			0.7 0.0571813207673312
	  			0.8 0.0996039477008501
	  			0.9 0.138418387694386
	  			1.01 0.760957676241099
	  		};
	  		\addlegendentry{$\mathtt{tol} (a_2)$}
	  		\addplot [semithick, green!50.1960784313725!black, mark=*, mark size=2, mark options={solid}]
	  		table {%
	  			-0.01 0.00216319975988598
	  			0.1 0.0133576912757673
	  			0.2 0.0133576912757673
	  			0.3 0.0133576912757673
	  			0.4 0.0133576912757673
	  			0.5 0.0133576912757673
	  			0.6 0.0133576912757673
	  			0.7 0.0133576912757673
	  			0.8 0.0131753087501526
	  			0.9 0.0215040678106359
	  			1.01 0.116110582362147
	  		};
	  		\addlegendentry{$\mathtt{err} (a_3)$}
	  		\addplot [semithick, green!50.1960784313725!black, dashed]
	  		table {%
	  			-0.01 0.0172376828230601
	  			0.1 0.0191989292210158
	  			0.2 0.0209276418403233
	  			0.3 0.0226563544596309
	  			0.4 0.0243850670789385
	  			0.5 0.0261653891136359
	  			0.6 0.0279520801616294
	  			0.7 0.0297387712096229
	  			0.8 0.0316811031373093
	  			0.9 0.0494783326147054
	  			1.01 0.275756139706156
	  		};
	  		\addlegendentry{$\mathtt{tol} (a_3)$}
	  	\end{axis}
	  	
	  \end{tikzpicture}
	\end{minipage}%
	\caption{Second experiment in \Cref{subsec:exp2}: spatial convergence for fixed tolerance factor $0.5$ (left) and error as well as maximal tolerance in dependence on tolerance factor (right).}
	\label{fig:exp2b}
\end{figure}

On the whole, we conclude that in all cases considered, we could achieve relative energy errors of only a few percent already using moderate mesh widths and time step sizes and updating only about half of the correctors on average in every time step.

\section*{Conclusion}
In this work, we analyzed a multiscale method in the spirit of the localized orthogonal decomposition for wave equations with time-dependent multiscale coefficients.
The method constructs (coarse) multiscale spaces in each time step.
We rigorously proved convergence rates in the mesh width and the time step size for spatially rough coefficients.
For this, we showed the exponential decay of the time derivative of the multiscale basis functions.
To obtain a computationally efficient method, we proposed an adaptive update strategy for the multiscale basis based upon an appropriate error indicator.
The presented numerical examples have underlined the theoretical findings and in particular show that small updates in every time step and moderate choices of the oversampling parameter are already sufficient to obtain reasonable approximations.

We expect the methodology and the techniques of error analysis to carry over to other problem classes as well, for instance parabolic problems with space- and time-dependent coefficients.
Further, by considering several previous time steps in the adaptive update strategy, even more coefficient classes such as (almost) time-periodic coefficients may become treatable in the future.


\end{document}